\documentclass{amsart}
\usepackage[margin=1in]{geometry}

\usepackage{tikz}
\usepackage{physics}
\usepackage{mathdots}
\usepackage{yhmath}
\usepackage{cancel}
\usepackage{siunitx}
\usepackage{array}
\usepackage{multirow}
\usepackage{wrapfig}
\usepackage[font=small]{caption}
\usepackage[all]{xy}
\usepackage{amscd}
\usepackage{gensymb}
\usepackage{tabularx}
\usepackage{extarrows}
\usepackage{booktabs}
\usetikzlibrary{fadings}
\usetikzlibrary{patterns}
\usetikzlibrary{shadows.blur}
\usetikzlibrary{shapes}
\usepackage{amsmath}
\usepackage{amsthm}
\usepackage{amssymb}
\usepackage{mathtools}
\usepackage{soul}
\usepackage{color}

\usepackage{ytableau}
\usepackage{shuffle}

\newtheorem{thm}{Theorem}[section]

\newtheorem{prop}[thm]{Proposition}
\newtheorem{cor}[thm]{Corollary}

\theoremstyle{definition}
\newtheorem{defn}[thm]{Definition}
\newtheorem{ex}[thm]{Example}
\newtheorem{rem}[thm]{Remark}

\newcommand{\psent}{p-sentence }
\newcommand{\psents}{p-sentences }

\newcommand{\NSym}{\textit{NSym}}
\newcommand{\QSym}{\textit{QSym}}
\newcommand{\smallersetminus}{\mathbin{\vcenter{\hbox{$\scriptscriptstyle\mathrlap{\setminus}{\hspace{.2pt}}\setminus$} }}}

\def\shuffle{\begin{picture}(16,5)(-2,0) 
\put(0,0){\line(0,1){5}}
\put(5,0){\line(0,1){5}}
\put(10,0){\line(0,1){5}}
\put(0,0){\line(1,0){10}}
\end{picture}}
\def\Qshuffle{\stackrel{Q}{\shuffle}}

\title[Colored symmetric functions]{ A Hopf algebra generalization of the symmetric functions in partially commutative variables}
\author[S. Daugherty]{Spencer Daugherty} \thanks{spencer.daugherty@colorado.edu. Department of Mathematics, University of Colorado Boulder, USA}

\usepackage[backend=bibtex, maxbibnames=10, doi=false, url=false, isbn=false]{biblatex}
\begin{document}

\begin{abstract} The quasisymmetric functions, $QSym$, are generalized for a finite alphabet $A$ by the colored quasisymmetric functions, $QSym_A$, in partially commutative variables. Their dual, $NSym_A$, generalizes the noncommutative symmetric functions, $NSym$, through a relationship with a Hopf algebra of trees.  We define an algebra $Sym_A$, contained within $QSym_A$, that is isomorphic to the symmetric functions, $Sym$, when $A$ is an alphabet of size one. We show that $Sym_A$ is a Hopf algebra and define its graded dual, $PSym_A$, which is the commutative image of $NSym_A$ and also generalizes $Sym$. The seven algebras listed here can be placed in a commutative diagram connected by Hopf morphisms. In addition to defining generalizations of the classic bases of the symmetric functions to $Sym_A$ and $PSym_A$, we describe multiplication, comultiplication, and the antipode in terms of a basis for both algebras. We conclude by defining a pair of dual bases that generalize the Schur functions and listing open questions. 
\end{abstract}

\maketitle

\section{Introduction}

Generalizations of the symmetric functions are widely studied and have a variety of formulations that lift different properties and structures. Among these generalizations are the quasisymmetric functions and noncommutative functions ($NSym$), which are significant to algebra and combinatorics in their own right. Along with $Sym$, both $QSym$ and $NSym$ are Hopf algebras - as are many other generalizations of the symmetric functions, or even generalizations of the quasisymmetric and noncommutative functions. These constitute a rich and varied subject of study.

For example, in \cite{Duchamp}, $Sym$ is generalized to the free symmetric functions $FSym$ while $QSym$ is generalized to the free quasisymmetric functions $FQSym$ and the matrix quasisymmetric functions $MQSym$, all of which are Hopf algebras. The preceding algebras have bases indexed by combinatorial objects such as standard young tableaux, permutations, and packed integer matrices as opposed to compositions or partitions, and $FQSym$ is isomorphic to the well-known Malvenuto-Reutenauer algebra \cite{Duch2, MalRet}.  As another example, symmetric functions in superspace, which incorporate noncommuting variables, were developed in \cite{Desrosiers} and shown to constitute a cocommutative and self-dual hopf algebra in \cite{fishel}. The latter paper also introduces noncommutative symmetric functions in superspace and quasisymmetric functions in superspace, which they show to be dual hopf algebras. For more Hopf algebra generalizations of symmetric functions, see \cite{Guo_MR, hazel_word, MalRet}.

Other generalizations of the symmetric functions do not focus on the Hopf algebra structure.
In \cite{Jing}, generalizations of the symmetric functions using vertex operators and their deformations are studied, with special attention paid to the Hall-Littlewood functions and raising operator formulas. \cite{Vaccarino} generalizes the symmetric functions by studying polynomial invariants of multiple (as opposed to single in the classical case) matrices by conjugation of the general linear group. Even more generalizations of symmetric (or quasisymmetric) functions can be found in \cite{Ahmia, Chrz2,Chrz1,gessel_enum,color_thesis,poirier,verdestar}.

Our initial focus is on the Hopf algebra generalizations of the quasisymmetric and noncommutative symmetric functions introduced by Doliwa \cite{doliwa21} called $QSym_A$ and $NSym_A$, respectively, defined for an alphabet $A$. In the case that $A$ is an alphabet of size one, $QSym_A$ and $NSym_A$ are isomorphic to $QSym$ and $NSym$, respectively.  The origin of these generalizations, specifically of $NSym_A$, lies in a Hopf algebra of rooted trees. Such Hopf algebras have been studied intently, in large part for their applications to the renormalization process in quantum field theory \cite{Foissy2013AnIT}. 

The Connes-Kreimer Hopf algebra of rooted trees, $CK$, was introduced in \cite{Connes_1998} while its noncommutative version on planar rooted trees, $NCK$, was introduced in \cite{Foissy}. 
These are related closely to other well-known Hopf algebras, for example, the graded dual of $CK$ is the Grossman-Larson Hopf algebra, $GL$, introduced in \cite{GROSSMAN1989184}. 
Additionally, $NCK$ is isomorphic to the Loday-Ronco Hopf algebra, $LR$, and the Brouder-Frabetti Hopf algebra, $BF$ \cite{Foissy, Hivert, holt, vanderlaan2002hopf}. 
Aguiar and Sottile \cite{aguiar2006structure} and Hoffman \cite{hoffman2007noncommutative} study the relationships between these algebras of trees and $QSym$ and $NSym$. 
Specifically, $NSym$ is shown to be isomorphic to a Hopf subalgebra of $NCK$ while $QSym$ is isomorphic to a Hopf subalgebra of the graded dual of $NCK$,  which turns out to be $NCK$ itself. $NSym$ is isomorphic to the subalgebra of $NCK$ made up of tall rooted planar trees, and $Sym$ is isomorphic to the subalgebra of $CK$ made up of tall rooted trees. Doliwa's definitions of $NSym_A$ and $QSym_A$ generalize these relationships using decorated trees from Foissy's decorated versions of $CK$ and $NCK$ \cite{Foissy}.  Doliwa \cite{doliwa21} studies the decorated $NCK$ but reformulated as a Hopf algebra of rooted ordered colored trees, $ROC$. Specifically, he identifies the Hopf subalgebra of tall rooted ordered colored trees and defines $NSym_A$ so that it is isomorphic.  

In this paper, we introduce and study two dual Hopf algebras, $Sym_A$ and $PSym_A$, that generalize the symmetric functions and specialize $QSym_A$ and $NSym_A$. Specifically, $PSym_A$ is the commutative image of $NSym_A$ and $Sym_A$ is a subalgebra of $QSym_A$. In Section \ref{sec:background}, we begin with an overview of symmetric functions following \cite{EC2, Zab} and Hopf algebras based on \cite{doliwa21, grinberg}, which leads into material on the quasisymmetric and noncommutative symmetric functions from the previous two sources as well as \cite{berg18, mason}. We close our background section by presenting Doliwa's colored generalizations of $QSym$ and $\NSym$ \cite{doliwa21}. 

In Section \ref{sec:psyma}, we define the algebra $PSym_A$ using commutative generators associated with words of alphabet $A$, which yield basis elements indexed by sorted sentences. We describe a multiplication, comultiplication, counit, and antipode that give $PSym_A$ a Hopf algebra structure. Finally, we show that $PSym_A$ is the commutative image of $NSym_A$ and, when $A$ is size one, isomorphic to $Sym$.

In Section \ref{sec:syma}, we define $Sym_A$ and show that it is a Hopf subalgebra of $QSym_A$. We express multiplication, comultiplication, and the antipode in terms of the colored monomial basis of $Sym_A$.  Like $PSym_A$, we show that $Sym_A$ is isomorphic to $Sym$ when $A$ is size one. Moreover, we prove that $Sym_A$ and $PSym_A$ are dual Hopf algebras.

In Section \ref{sec:colorschur} by introducing colored generalizations of the Schur basis in $Sym_A$ and $PSym_A$. We define the colored dual Schur symmetric functions in $Sym_A$ in terms of their expansion in the colored monomial basis and colored semistandard Young tableaux. The colored Schur bases in $PSym_A$ are defined as their duals. 

We close with open questions, largely relating to the colored Schur-like bases of $NSym_A$ and $QSym_A$ introduced by the author in \cite{Dau23, Dau_thesis}.

\subsection{Acknowledgements} The work in this paper also appeared in the author's doctoral thesis \cite{Dau_thesis}. Many thanks to Laura Colmenarejo and Sarah Mason for their support on this project.

\section{Background}\label{sec:background}

A finite sequence of positive integers $\alpha = (\alpha_1, \alpha_2, \ldots, \alpha_k)$ is a \emph{composition of} $n$ , written $\alpha \models n$, if $\sum_{i=1}^k \alpha_i = n$.  A finite sequence of positive integers $\lambda = (\lambda_1, \lambda_2, \ldots, \lambda_k)$ is a \emph{partition of} $n$ , written $\lambda \vdash n$, if $\sum_{i=1}^k \lambda_i = n$ and $\lambda_1 \geq \lambda_2 \geq \ldots \geq \lambda_k$. We let $sort(\alpha)$ denote the partition with the same entries as a composition $\alpha$. We say $\alpha \preceq \beta$, or $\alpha$ is a refinement of $\beta$, if $\{\beta_1, \beta_1+\beta_2, \beta_1+\beta_2 + \beta_3, \cdots\} \subseteq \{\alpha_1, \alpha_1+\alpha_2, \alpha_1+\alpha_2+\alpha_3, \cdots\}$.

Partitions and compositions are represented visually (in English notation) with diagrams made up of top-left aligned rows of boxes. The diagram of a partition $\alpha = (\alpha_1, \ldots, \alpha_k)$ will have $\alpha_1$ boxes in row $1$ (the top row), $\alpha_2$ boxes in row $2$, and so on. 

\subsection{Symmetric Functions}

A \textit{symmetric function} $f(x)$ with rational coefficients is a formal power series $f(x) = \sum_{\alpha} c_{\alpha}x^{\alpha}$ where $\alpha$ is a weak composition,  $x^{\alpha} = x_1^{\alpha_1}\ldots x_k^{\alpha_k}$, and $f(x_{\omega(1)}, x_{\omega(2),\ldots}) = f(x_1,x_2,\ldots)$ for all permutations $\omega$ of $\mathbb{Z}_{>0}$. Here we take $x = (x_1, x_2, \ldots)$ and $c_{\alpha} \in \mathbb{Q}$. The algebra of symmetric functions is denoted $Sym$, and we choose to take $\mathbb{Q}$ as our base field. 


\begin{rem}
    We use a common convention and write the symmetric function $f(x)$ as $f$ without the set of variables and $(x)$. Each of these functions is over the variables $(x_1, x_2, \ldots)$.
\end{rem}

Now we present the classic bases of $Sym$.
For a partition $\lambda = (\lambda_1, \ldots, \lambda_k) \vdash n$, the \emph{monomial symmetric function} is defined as $$m_{\lambda} = \sum_{\alpha} x^{\alpha},$$ where the sum runs over all unique compositions $\alpha$ that are permutations of the entries of $\lambda$. 
The \emph{elementary symmetric function} $e_{\lambda}$ is defined by $$e_n = m_{1^n}= \sum_{i_1 < \cdots < i_n} x_{i_1} \cdots x_{i_n}, \text{\quad with \quad} e_{\lambda} = e_{\lambda_1}e_{\lambda_2} \cdots e_{\lambda_k}.$$
The \emph{complete homogeneous symmetric function} is defined by 
$$h_n = \sum_{\lambda \vdash n} m_{\lambda} = \sum_{i_1 \leq \cdots \leq i_n}x_{i_1} \cdots x_{i_n} \text{\quad with \quad} h_{\lambda} = h_{\lambda_1} \cdots h_{\lambda_k}.$$
The sets $\{m_{\lambda}\}_{\lambda}$, $\{h_{\lambda}\}_{\lambda}$, and $\{e_{\lambda}\}_{\lambda}$ are bases of $Sym$. Two bases of $Sym$ $\{u_{\lambda}\}_{\lambda}$ and $\{v_{\lambda}\}_{\lambda}$ are \emph{dual bases} if and only if $\langle u_{\lambda}, v_{\mu} \rangle = \delta_{\lambda, \mu}$, using a bilinear form denoted by $\langle \cdot, \cdot \rangle: Sym \times Sym \rightarrow \mathbb{Q}$ where $\langle m_{\lambda}, h_{\mu} \rangle = \delta_{\lambda, \mu}$. $Sym$ is a self-dual algebra under this inner product.
There is a classical involutive endomorphism on the symmetric functions $\omega: Sym \rightarrow Sym$ defined by $\omega(e_{\lambda}) = h_{\lambda}$.   The involution $\omega$ is invariant under duality, which is to say $\langle \omega(f), \omega(g) \rangle = \langle f,g \rangle$ for all $f,g \in Sym$.

Given a partition $\lambda \vdash n$, a \emph{semistandard Young tableau} (SSYT) of shape $\lambda$ is a filling of the Young diagram of $\lambda$ filled with positive integers such that the entries in each row weakly increase from left to right, and the entries in each column strictly increasing from top to bottom. The \emph{size} of a SSYT is $|\lambda|$, the number of boxes, and its \emph{type} (or content) is a weak composition that encodes the number of times each integer appears as an entry in the tableau. We write $type(T) = (\beta_1,\ldots,\beta_j)$ if $T$ has $\beta_i$ boxes containing an $i$ for all $i \in [j]$. A \emph{standard Young tableau} (SYT) of size $n$ is a Young tableau in which each number in $ [n]$ appears exactly once.  A tableaux $T$ of type $\beta = (\beta_1, \ldots, \beta_k)$ is associated with the monomial $x^T = x_1^{\beta_1} \cdots x_k^{\beta_k}$.

\begin{ex} The semistandard Young tableaux of shape $(2,2)$ with entries in $\{1,2,3\}$ and their associated monomials are:
$$\begin{ytableau}1&1\\2&2\end{ytableau} \quad \quad
\begin{ytableau}1&1\\2&3\end{ytableau} \quad \quad
\begin{ytableau}1&1\\3&3\end{ytableau} \quad \quad
\begin{ytableau}1&2\\2&3\end{ytableau} \quad \quad
\begin{ytableau}1&2\\3&3\end{ytableau} \quad \quad
\begin{ytableau}2&2\\3&3\end{ytableau}
$$
$$\ \ x_1^2x_2^2 \quad \quad x_1^2x_2x_3  \quad \quad \ \ x_1^2x_3^2 \quad \quad x_1x_2^2x_3 \quad \quad \ x_1x_2x_3^2  \quad \ \ \ x_2^2x_3^2 \quad$$
The standard Young tableaux of shape $(2,2)$, both of type $(1,1,1,1)$, are:

$$\begin{ytableau}1&2\\ 3&4\end{ytableau} \quad \quad \begin{ytableau}1&3\\ 2&4\end{ytableau}$$

\end{ex}\vspace{-1mm}

Semistandard Young tableaux are the combinatorial objects used to define the Schur symmetric functions. These functions appear throughout combinatorics, representation theory, algebraic geometry, and other fields, and form a crucial part of the foundation of the theory of symmetric functions.

For a partition $\lambda$, the \emph{Schur symmetric function} is defined as $$s_{\lambda} = \sum_T x^T,$$ where the sum runs over all semistandard Young tableaux $T$ of shape $\lambda$ with entries in $\mathbb{Z}_{> 0}$.


The Schur functions are an orthonormal basis for $Sym$, meaning $\langle s_{\lambda}, s_{\mu} \rangle = \delta_{\lambda, \mu}$. Additionally, the Schur basis maps to itself under $\omega$ by $\omega(s_{\lambda}) = s_{\lambda'}$. Define the \emph{Kostka number} $K_{\lambda, \alpha}$ as the number of $SSYT$s of shape $\lambda$ and type $\alpha$. Then, $$s_{\lambda} = \sum_{\mu}K_{\lambda, \mu}m_{\mu} \text{\qquad and \qquad} h_{\mu} = \sum_{\lambda} K_{\lambda, \mu} s_{\lambda}.$$  

The Kostka numbers have a variety of interesting combinatorics and applications, see \cite{Alexandersson2020} for an overview of topics and sources.
n{Hopf algebras}\label{hopf_alg_sec}
Hopf algebras are important structures in combinatorics as well as algebraic topology, representation theory, physics, and quantum field theory, with notable examples including $Sym$, $QSym$, and $\NSym$. See \cite{doliwa21, grinberg} for more details.

 Let $\Bbbk$ be a field of characteristic zero. An \emph{associative algebra} is a $\Bbbk$-module $\mathcal{H}$ with $\Bbbk$-linear multiplication $\mu: \mathcal{H} \otimes \mathcal{H} \rightarrow \mathcal{H}$ and a $\Bbbk$-linear unit $\eta: \Bbbk \rightarrow \mathcal{H}$, for which the following diagrams commute:
 $$  \begin{CD}
            \mathcal{H}\otimes \mathcal{H}\otimes \mathcal{H} @ > \mathrm{id} \otimes \mu >>  \mathcal{H} \otimes \mathcal{H} \\
            @V \mu \otimes \mathrm{id} VV     @VV \mu V \\
            \mathcal{H}\otimes \mathcal{H} @ > \mu >>  \mathcal{H}
            \end{CD}
            \qquad \qquad
            \begin{CD}
            \mathcal{H}\otimes \Bbbk @= \mathcal{H}  @= \Bbbk \otimes \mathcal{H}\\
            @V\mathrm{id} \otimes \eta VV @V\mathrm{id} VV @VV \eta \otimes \mathrm{id} V \\
            \mathcal{H}\otimes \mathcal{H} @> \mu >>  \mathcal{H} @< \mu << \mathcal{H}\otimes \mathcal{H}
\end{CD}$$
A \emph{co-associative coalgebra} is a $\Bbbk$-module $\mathcal{H}$ with $\Bbbk$-linear comultiplication $\Delta: \mathcal{H} \rightarrow \mathcal{H} \otimes \mathcal{H}$ and $\Bbbk$-linear counit $\epsilon: \mathcal{H} \rightarrow \Bbbk$, satisfying the commutative diagrams:
 $$
        \begin{CD}
        \mathcal{H} @ > \Delta >> \mathcal{H}\otimes \mathcal{H} \\
        @V \Delta  VV     @VV \Delta \otimes \mathrm{id} V \\
        \mathcal{H} \otimes \mathcal{H} @ > \mathrm{id} \otimes \Delta >> \mathcal{H}\otimes \mathcal{H}\otimes \mathcal{H} 
        \end{CD} \qquad \qquad
        \begin{CD}
        \mathcal{H}\otimes \mathcal{H} @ < \Delta <<  \mathcal{H} @ > \Delta >> \mathcal{H}\otimes \mathcal{H} \\
        @V\mathrm{id} \otimes \epsilon VV @V\mathrm{id} VV @VV \epsilon \otimes \mathrm{id} V \\
        \mathcal{H}\otimes \Bbbk @= \mathcal{H}  @= \Bbbk \otimes \mathcal{H}
        \end{CD}
     $$

Note that in a standard abuse of notation, we often leave out the formal notation for multiplication and just write $b_1 b_2$ to indicate the multiplication of $b_1$ and $b_2$, or $\mu(b_1 \otimes b_2)$. 

An \emph{algebra (homo)morphism} is a $\Bbbk$-linear map $\mathcal{A} \xrightarrow{\varphi} \mathcal{B}$ between two algebras $\mathcal{A}$ and $\mathcal{B}$ such that for $a_1 ,a_2 \in \mathcal{A}$ and $k \in \Bbbk$, $$(\varphi \circ \mu_{\mathcal{A}})(a_1 \otimes a_2) = (\mu_{\mathcal{B}} \circ (\varphi \otimes \varphi))(a_1 \otimes a_2) \text{\quad and \quad} (\varphi \circ \eta_{\mathcal{A}})(k) = \eta_{\mathcal{B}}(k).$$
A \emph{coalgebra (homo)morphism} is a linear map $\mathcal{C} \xrightarrow{\varphi} \mathcal{D}$ between two coalgebras $\mathcal{C}$ and $\mathcal{D}$ such that for $c \in \mathcal{C}$, \begin{equation}\label{coalg_morph} 
(\Delta_{\mathcal{D}} \circ \varphi)(c) = ((\varphi \otimes \varphi) \circ \Delta_{\mathcal{C}})(c) \text{\quad and \quad} (\epsilon_{\mathcal{D}} \circ \varphi)(c) = \epsilon_{\mathcal{C}}(c).
\end{equation}

The subspace $\mathcal{G}$ is a  \emph{subalgebra} of $\mathcal{H}$ if $\mu$ restricts to $\mathcal{G}$, and a \emph{subcoalgebra} of $\mathcal{H}$ if $\Delta$ restricts to $\mathcal{G}$.

\begin{defn}\label{bialgebra}
    $(\mathcal{H}, \mu, \Delta, \eta, \epsilon)$ is a \emph{bialgebra} if $(\mathcal{B}, \mu, \eta)$ is an algebra and $(\mathcal{H}, \Delta, \epsilon)$ is a coalgebra and the following hold for $T(x \otimes y) = y \otimes x$ where $k \in \Bbbk$ and $h_1, h_2 \in \mathcal{H}$:
    \begin{enumerate}
        \item $(\Delta \circ \mu)(h_1 \otimes h_2) = ((\mu \otimes \mu) \circ (id \otimes T \otimes id)\circ (\Delta \otimes \Delta))(h_1 \otimes h_2)$
        \item $(\mu \circ (\epsilon \otimes \epsilon))(h_1 \otimes h_2) = (\epsilon \circ \mu)(h_1 \otimes h_2)$
        \item $(\Delta \circ \eta)(k) =  ((\eta \otimes \eta) \circ \Delta)(k)$
        \item $id(k)= (\epsilon \circ \eta)(k)$.
    \end{enumerate}
\end{defn}

A bialgebra $\mathcal{H}$ is \emph{graded} if it is graded as a $\Bbbk$-module $\mathcal{H} = \bigoplus_{n \geq 0} \mathcal{H}^{(n)}$ with 
$$\mathcal{H}^{(n)} \otimes \mathcal{H}^{(m)} \xrightarrow{\mu} \mathcal{H}^{(n+m)}, \qquad \mathcal{H}^{(n)} \xrightarrow{\Delta} \bigoplus_{m+p=n} \mathcal{H}^{(m)} + \mathcal{H}^{(p)}.$$ 
A graded bialgebra is \emph{connected} if $\mathcal{H}^{(0)} \cong \Bbbk$. 

\begin{defn}\label{hopf_defn} A bialgebra $(\mathcal{H}, \mu, \Delta, \eta, \epsilon)$ is a \emph{Hopf algebra} if there exists a $\Bbbk$-linear anti-endomorphism on $\mathcal{H}$, called the \emph{antipode}, such that the following diagram commutes:
 $$
    \label{antipode-diagram}
    \xymatrix{
    &\mathcal{H} \otimes \mathcal{H} \ar[rr]^{S \otimes id_\mathcal{H}}& &\mathcal{H} \otimes \mathcal{H} \ar[dr]^\mu& \\
    \mathcal{H} \ar[ur]^\Delta \ar[rr]^{\epsilon} \ar[dr]_\Delta& &\Bbbk \ar[rr]^{\eta} & & \mathcal{H}\\
    &\mathcal{H} \otimes \mathcal{H} \ar[rr]_{id_\mathcal{H} \otimes S}& &\mathcal{H} \otimes \mathcal{H} \ar[ur]_\mu& \\ }
    $$
    
\end{defn}

\begin{prop}\cite{ MilM}\label{bi_to_hopf}
    Any graded and connected bialgebra is a Hopf algebra.
\end{prop}

Note that if $\mathcal{H}$ is a graded Hopf algebra, and $x\in \mathcal{H}^{(n)}$, then the antipode is expressed as 
\begin{equation}\label{anti_comult}
S(x) = - \sum_{i} S(y_i) z_i = - \sum_{i} y_i S(z_i) \end{equation}
where $\Delta(x) = \sum_i y_i \otimes z_i$ (Sweedler notation).

A subalgebra $\mathcal{G}$ of $\mathcal{H}$ is a \emph{Hopf subalgebra} if $\mathcal{G}$ is a subcoalgebra of $\mathcal{H}$ and the antipode $S$ of $\mathcal{H}$ restricts to $\mathcal{G}$. A \emph{bialgebra (homo)morphism} is a linear map between two bialgebras that is both an algebra homomorphism and a coalgebra homomorphism.

\begin{cor}\cite{grinberg}\label{hopf_morph}
    Let $H_1$ and $H_2$ be Hopf algebras with antipodes $S_1$ and $S_2$, respectively. Then, any bialgebra morphism $H_1 \xrightarrow{\beta} H_s$ is a Hopf morphism, that is, it commutes with the antipodes $(\beta \circ S_1 = S_2 \circ \beta)$.
\end{cor}

Duality between Hopf algebras is one of the most important tools for working with our chosen algebras. 

\begin{defn} \label{hopf_dual_mult}
Let $(\mathcal{A},  \mu_{\mathcal{A}}, \Delta_{\mathcal{A}}, \eta_{\mathcal{A}}, \epsilon_{\mathcal{A}})$ and $(\mathcal{B},  \mu_{\mathcal{B}}, \Delta_{\mathcal{B}}, \eta_{\mathcal{B}}, \epsilon_{\mathcal{B}})$ be two Hopf Algebras with antipodes $\mathcal{S}_{\mathcal{A}}$ and $\mathcal{S}_{\mathcal{B}}$ respectively, where $1_{\mathcal{A}}$ and $1_{\mathcal{B}}$ are the multiplicative identity elements. We say that $\mathcal{A}$ and $\mathcal{B}$ are \emph{dually paired} by an inner product $\langle \ ,\ \rangle: \mathcal{B} \otimes \mathcal{A} \rightarrow \mathbb{Q}$, if, for all elements $a, a_1, a_2 \in \mathcal{A}$ and $b, b_1, b_2 \in \mathcal{B}$, we have $$\langle \mu_{\mathcal{B}}(b_1,b_2), a \rangle = \langle b_1 \otimes_{\mathcal{B}} b_2, \Delta_{\mathcal{A}}(a)\rangle, \qquad \qquad \langle 1_{\mathcal{B}}, a \rangle = \epsilon_{\mathcal{A}}(a), \qquad \qquad \qquad \qquad \qquad \qquad \qquad$$ 
$$\langle b, \mu_{\mathcal{A}}(a_1,a_2) \rangle = \langle \Delta_{\mathcal{B}}(b), a_1 \otimes_{\mathcal{A}} a_2 \rangle, \qquad \qquad \epsilon_{\mathcal{B}}(b) = \langle b, 1_{\mathcal{A}}  \rangle, \qquad \qquad \langle S_{\mathcal{B}}(b), a \rangle = \langle b, S_{\mathcal{A}}(a) \rangle. $$
\end{defn} 
Two bases $\{a_i\}_{i \in I}$ and $\{b_i\}_{i \in I}$ of $\mathcal{A}$ and $\mathcal{B}$ respectively are \emph{dual bases} if and only if $\langle a_i, b_i \rangle = \delta_{i,j}$. We say that a map $f: \mathcal{A} \rightarrow \mathcal{A}$ is \emph{adjoint} to a map $g: \mathcal{B} \rightarrow \mathcal{B}$ if \begin{equation}\label{adjoint}
\langle f(b_1), a_1 \rangle = \langle b_1, g(a_1) \rangle
\end{equation} for any $a_1 \ \in \mathcal{A}, b_1 \in \mathcal{B}$. The following result gives a relation for the change of bases using duality. 

\begin{prop}\label{hopf_dual} \cite{Hoffman} Let $\mathcal{A}$ and $\mathcal{B}$ be dually paired algebras and let $\{a_i\}_{i \in I}$ be a basis of $\mathcal{A}$. A basis $\{b_i\}_{i \in I}$ of $\mathcal{B}$ is the unique basis that is dual to $\{a_i\}_{i \in I}$ if and only if the following relationship holds for any pair of dual bases $\{c_i\}_{i \in I}$ in $\mathcal{A}$ and $\{d_i\}_{i \in I}$ in $\mathcal{B}$: $$a_i = \sum_{j\in I} k_{i,j} c_j \quad \Longleftrightarrow \quad d_j = \sum_{i \in I} k_{i,j} b_i.$$
\end{prop}

The coefficients of the multiplication and comultiplication of dual bases have a similar relationship.

\begin{prop}\label{hopf_dual_coproduct}\cite{grinberg}
The comultiplication of the basis $\{b_i\}_{i \in I}$ in $\mathcal{B}$ is uniquely defined by the multiplication of its dual basis $\{a_i\}_{i \in I}$ in $\mathcal{A}$ in the following way:  $$a_ja_k = \sum_{i \in I}c^{i}_{j,k}a_i \quad \Longleftrightarrow \quad \Delta(b_i) = \sum_{(j,k) \in I \times I}c^{i}_{j,k}b_j \otimes b_k.$$ Further, $\Delta: \mathcal{B} \rightarrow \mathcal{B} \otimes \mathcal{B}$ is an algebra homomorphism.
\end{prop}


We conclude this subsection by reviewing the Hopf algebra structure of $Sym$ from \cite{grinberg} and \cite{Zab}. Multiplication  on the monomial symmetric functions is expressed, for $\lambda \vdash n$ and $\mu \vdash k$, as \begin{equation}\label{mon_mult}
    m_{\lambda} m_{\mu} = \sum_{\nu \vdash n + k} r^{\nu}_{\lambda, \mu} m_{\nu}
\end{equation} where $r^{\nu}_{\lambda, \mu}$ is the number of pairs of sequences $(\alpha, \beta)$ with $\alpha_i, \beta_i \geq 0$ where $sort(\alpha) = \lambda$ and $sort(\beta)= \mu$ where $\nu = (\alpha_1 + \beta_1, \alpha_2 + \beta_2, \ldots)$.  Comultiplication is given by \begin{equation}\label{mon_comult}
    \Delta(m_{\lambda}) = \sum_{\mu \sqcup \nu = \lambda} m_{\mu} \otimes m_{\nu},
\end{equation} where $\mu \sqcup \nu$ the multiset union of the parts of $\mu$ and $\nu$ sorted into a partition. Multiplication on the complete homogeneous symmetric functions is expressed as 
\begin{equation}\label{h_mult}
    h_{\lambda}h_{\mu} = h_{sort(\lambda \cdot \mu)},
\end{equation} for partitions $\lambda$ and $\mu$. Comultiplication is expressed as 
\begin{equation}\label{h_comult}
\Delta(h_n) = \sum_{k=0}^n h_k \otimes h_{n-k}.
\end{equation}

\subsection{Quasisymmetric functions}
Given a weak composition $\alpha = (\alpha_1, \alpha_2, \ldots)$, write $x^{\alpha} = x_1^{\alpha_1} x_2 ^{\alpha_2} \cdots$. If $\alpha$ has $k$ non-zero entries given by $\alpha_{i_1} = a_1$, $\alpha_{i_2} = a_2$, $\ldots$, $\alpha_{i_k} = a_k$ with $i_1 < \cdots < i_k $, then $x^{\alpha} = x_{i_1}^{a_1}x_{i_2}^{a_2} \cdots x_{i_k}^{a_k}$.  A \emph{quasisymmetric function} $f(x)$ is a formal power series  
$$f(x) = \sum_{\alpha} b_{\alpha}x^{\alpha},$$ where the sum runs over weak compositions $\alpha$ and the coefficients of the monomials $x_{i_1}^{a_1}\ldots x_{i_k}^{a_k}$ and $x_{j_1}^{a_1}\ldots x_{j_k}^{a_k}$ are equal if $i_1 < \ldots  < i_k$ and $j_1 < \ldots  < j_k$. 

For a composition $\alpha$, the \textit{monomial quasisymmetric function} $M_{\alpha}$ is defined by $$M_{\alpha} = \sum_{i_1<\ldots <i_k}x_{i_1}^{\alpha_1}\ldots x_{i_k}^{\alpha_k},$$ where the sum runs over strictly increasing sequences of $k$ positive integers $i_1, \ldots, i_k \in \mathbb{Z}_{>0}$.
The \textit{fundamental quasisymmetric function} $F_{\alpha}$ is defined as $$F_{\alpha} = \sum_{\beta \preceq \alpha}M_{\beta}.$$ Then, $M_{\alpha} = \sum_{\beta \preceq \alpha} (-1)^{\ell(\alpha)-\ell(\beta)} F_{\beta}$ by the M\"{o}bius inversion.  
\begin{ex} The monomial quasisymmetric function indexed by $(1,2)$ is
$$M_{(1,2)} = \sum_{i<j}x_ix^2_j = x_1x^2_2 + x_1x^2_3 + \ldots  + x_2x^2_3 + x_2x^2_4 + \ldots  + x_3x^2_4 + x_3x^2_5 + \ldots.$$ The expansion of $F_{(3)}$ into the monomial basis is
$$F_{(3)} = M_{(3)}+M_{(2,1)}+M_{(1,2)} + M_{(1,1,1)}.$$
\end{ex}

The quasisymmetric functions form a Hopf algebra denoted $QSym$. The monomial basis has multiplication and comultiplication inherited from the quasishuffle and concatenation operations on compositions. The \emph{quasishuffle} $ \Qshuffle$ of compositions is defined as the sum of shuffles of $\alpha = (\alpha_1, \ldots , \alpha_k)$ and $\beta = (\beta_1, \ldots , \beta_l)$ where any consecutive pairs $\alpha_i$ and $\beta_j$ (in that order) may be replaced with $\alpha_i + \beta_j$.  Note that the same composition may appear multiple times in the quasishuffle. Multiplication of monomial functions is given by $$M_{\alpha}M_{\beta}= \sum_{\gamma}M_{\gamma},$$ where the sum runs over all $\gamma$ such that $\gamma$ is a summand in $\alpha \Qshuffle \beta$, with multiplicity. Comultiplication is expressed as $$\Delta(M_{\alpha}) = \sum_{\beta \cdot \gamma = \alpha}M_{\beta} \otimes M_{\gamma},$$ where the sum runs over all compositions $\beta, \gamma$ where $\beta \cdot \gamma = \alpha$.

\begin{ex} The following expressions show the multiplication and comultiplication on monomial quasisymmetric functions expanded in terms of the monomial basis:
$$M_{(1)}M_{(2,1)} = M_{(1,2,1)}+2M_{(2,1,1)}+M_{(3,1)}+M_{(2,2)},$$
$$\Delta(M_{(2,1,1)}) = 1 \otimes M_{(2,1,1)} + M_{(2)}\otimes M_{(1,1)} + M_{(2,1)} \otimes M_{(1)} + M_{(2,1,1)}\otimes 1.$$
\end{ex} 

For more on the quasisymmetric functions, see \cite{mason}.

\subsection{Noncommutative symmetric functions}
The Hopf algebra dual of $QSym$ is the algebra of \emph{noncommutative symmetric functions}, denoted by $\NSym$.  It can be defined as the algebra with noncommutative generators $\{H_1,H_2, \ldots \}$ and no relations,
$$\NSym = \mathbb{Q}  \left\langle H_1, H_2, \ldots \right\rangle.$$ 
For a composition $\alpha = (\alpha_1, \ldots, \alpha_k)$, define the \textit{complete homogeneous noncommutative symmetric function} as $H_{\alpha} = H_{\alpha_1}H_{\alpha_2} \ldots H_{\alpha_k}$. The set $\{H_{\alpha}\}_{\alpha}$ forms a basis of $\NSym$ called the \textit{complete homogeneous basis}. $\NSym$ and $QSym$ are dually paired by the inner product defined by $\langle H_{\alpha}, M_{\beta} \rangle = \delta_{\alpha, \beta}$ for all compositions $\alpha, \beta$. This inner product has all the properties listed in Definition \ref{hopf_dual_mult} and is the main tool for translating between $QSym$ and $\NSym$. We translate from $NSym$ to $Sym$ using the \emph{forgetful map} $\chi: \NSym \rightarrow Sym$, which is defined as  $$\chi(H_{\alpha}) = h_{sort(\alpha)},$$ extended linearly to map all elements in $\NSym$ to their commutative image in $Sym$. The forgetful map is a morphism that is adjoint to the inclusion of $Sym$ to $QSym$.

Given a composition $\alpha$, the \emph{ribbon noncommutative symmetric function} is defined as $$R_{\alpha}=\sum_{\alpha \preceq \beta }(-1)^{\ell{(a)}-\ell{(\beta)}}H_{\beta}, $$  and thus we can write
$ H_{\beta} = \sum_{\alpha \preceq \beta} R_{\alpha}.$ The ribbon functions, which are a basis of $NSym$, are dual to the fundamental basis of $QSym$, meaning $\langle R_{\alpha}, F_{\beta} \rangle = \delta_{\alpha, \beta}$. Multiplication in $\NSym$ on the ribbon and complete homogeneous functions is expressed as 
\begin{equation*}
    H_{\alpha} H_{\beta} = H_{\alpha \cdot \beta} \text{\qquad and \qquad} R_{\alpha}R_{\beta} = R_{\alpha \cdot \beta} + R_{\alpha \odot \beta}.
\end{equation*}
Given a composition $\alpha$, the \emph{elementary noncommutative symmetric function} is defined by $$E_{\alpha} = \sum_{\beta \preceq \alpha} (-1)^{|\alpha|-\ell(\beta)}H_{\beta},$$  and thus we can write $ H_{\beta} = \sum_{\alpha \preceq \beta} (-1)^{|\beta|- \ell(\alpha)} E_{\alpha}.$

For more details on the noncommutative symmetric functions see \cite{gelfand}.

\subsection{Doliwa's colored $\QSym_A$ and $\NSym_A$} \label{colorsection}
The algebra of noncommutative symmetric functions has a natural generalization to an algebra of sentences. Dually, the algebra of quasisymmetric functions has a generalization to partially commutative variables. These two algebras, $NSym_A$ and $QSym_A$, were introduced by Doliwa in \cite{doliwa21}.

Let $A = \{a_1, a_2, \ldots  , a_m\}$ be an alphabet of \textit{colors}.  We refer to finite sequences of colors in $A$ as words and write them without separating commas. 
We refer to finite sequences of non-empty words as \textit{sentences}. 
The $\emptyset$ denotes either the empty word or the empty sentence. 
A \emph{weak sentence} may include empty words, and is sometimes considered to be of infinite length but must have a finite number of non-empty words. 
The \textit{size} of a word $w$, denoted $|w|$, is the total number of colors it contains. 
Note that we are counting repeated colors here unless we say ``the number of unique colors''. 
The \textit{size} of a sentence $I=(w_1,w_2, \ldots , w_k)$, denoted $|I|$, is also the number of colors it contains.  
The \textit{length} of a sentence $I$, denoted $\ell(I)$, is the number of words in $I$. 
The \emph{concatenation} of two words $w = a_1\cdots a_k$ and $v=b_1\cdots b_j$ is given by $w \cdot v = a_1 \cdots a_kb_1 \cdots b_j$. This may also be denoted simply by $wv$.  
Concatenating every word in a sentence $I$ results in a word called the \textit{maximal word} of $I$, denoted $w(I)=w_1w_2\ldots w_k$. The \textit{word lengths} of $I$ are $w \ell(I)=(|w_1|, \ldots , |w_k|)$, which gives the underlying composition of the sentence.

\begin{ex} 
Let $a,b,c \in A$ and let $w_1 = cb$, $w_2 = aaa$, and $w_3 = c$ be words.  Consider the sentence $I = (w_1,w_2,w_3) = (cb,aaa,c)$.  Then, $|w_1|=2$, $|w_2|=3$, $|w_3|=1$, and $|I|=6$. The length of $I$ is $\ell(I)=3$ and the word length of $I$ is $w \ell(I)=(2,3,1)$.  The maximal word of $I$ is $w(I)=cbaaac$.  
\end{ex}

 A sentence $I$ is a refinement of a sentence $J$, written $ I \preceq J$, if $w(I)=w(J)$ and $w \ell(I) \preceq w \ell(J)$. In other words, if $J$ can be obtained by concatenating some adjacent words of $I$. In this case, $I$ is called a \emph{refinement} of $J$ and $J$ a \emph{coarsening} of $I$.

 Assume that $A$ has a total order $\leq$ and define the following \emph{lexicographic order} $\preceq_{\ell}$ on words. For words $w=a_1 \ldots a_k$ and $v=b_1 \ldots b_j$, we say $w \leq_{\ell} v$ if $a_i < b_i$ for the first positive integer $i$ such that $a_i \not= b_i$. If no such $i$ exists, then $w = v$. 

\begin{ex} Let $A = \{a < b < c\}$ and $I = (bac)$.  The refinements of $I$ are $(bac)$, $(b,ac)$, $(ba,c)$, and $(b,a,c)$. Under lexicographic order, $abc \preceq_{\ell} acb \preceq_{\ell} bac \preceq_{\ell} bca \preceq_{\ell} cab \preceq_{\ell} cba$.
\end{ex} 

The \emph{concatenation} of two sentences $I = (w_1, \ldots, w_k)$ and $J = (v_1, \ldots, v_h)$ is $I \cdot J = (w_1,\ldots ,w_k,v_1,\ldots ,v_h)$.  Their \emph{near-concatenation} is $I \odot J = (w_1,\ldots ,w_k \cdot v_1,\ldots ,v_h)$ where $w_k$ and $v_1$ are also concatenated. If $a_{i}$ is the $i^{\text{th}}$ entry in $I$ and $a_{i+1}$ is the $(i+1)^{\text{th}}$ entry in $I$, we say that $I$ \textit{splits} after the $i^{\text{th}}$ entry if $a_i \in w_j$ and $a_{i+1} \in w_{j+1}$ for some $j \in [k]$.

\begin{ex} Let $I = (bc,a)$ and $J = (b, ac)$.  Then, $I \cdot J = (bc,a,b,ac)$ and $I \odot J = (bc,ab,ac)$.  The sentence $(bc,a,b,ac)$ splits after the $2^{\text{nd}}$, $3^{\text{rd}}$ and $4^{\text{th}}$ entries.
\end{ex}

The \emph{reversal} of $I= (w_1,\ldots ,w_k)$ is $I^r = (w_k, w_{k-1},\ldots ,w_1)$. The \emph{complement} of $I$, denoted $I^c$, is the unique sentence such that $w(I)=w(I^c)$ and $I^c$ splits exactly where $I$ does not. Both maps are involutions on sentences.

\begin{ex}
Let $I = (ab,cde)$.  Then $I^r = (cde,ab)$ and $I^c = (a,bc,d,e).$
\end{ex}

The sentence obtained by removing all empty words from a weak sentence $J$, denoted by $\tilde{J}$, is called the \emph{flattening} of a sentence. Let a weak sentence $J = (v_1,\ldots ,v_k)$ and a sentence $I = (w_1,\ldots ,w_k)$. We say that $J$ is  \emph{right-contained} in $I$, or $J \subseteq_R I$, if there exists a weak sentence $I/_R J = (u_1,\ldots ,u_k)$ such that $w_i = u_iv_i$ for every $i \in [k]$. We say that $J$ is \emph{left-contained} in $I$, or $J \subseteq_L I$, if there exists a weak sentence $ I/_L J=(q_1, \ldots , q_k)$ such that $w_i = v_iq_i$ for every $i \in [k]$. Note that right-containment is denoted $I/J$ in \cite{doliwa21}.

\begin{ex}
Let $I = (abc,def)$, $J = ( c, ef)$, and $K = (a,de)$. Then $J \subseteq_R I$ and $I/_R J = (ab,d)$, while $K \subseteq_L I$ and $I/_L K = (c,f)$. Given the weak sentence $I = (\emptyset, ab, \emptyset, c)$, the flattening of $I$ is $\tilde{I}=(ab,c)$.
\end{ex}

\subsubsection{The Hopf algebra of sentences and colored noncommutative symmetric functions}
The algebra $\NSym_A$ is defined as the algebra freely generated over noncommuting elements $H_w$ for any word in $A$. There is a Hopf algebra structure on $NSym_A$ expressed over $\{H_I\}_I$ as follows:
$$H_I \cdot H_J = H_{I \cdot J}, \quad \quad \quad \Delta(H_I) = \sum_{J \subseteq_R I}H_{\widetilde{I/_R J}} \otimes H_{\tilde{J}}, \quad \quad \quad S(H_I) = \sum_{J \preceq I^r}(-1)^{\ell(J)}H_J.$$ 
The reversal and complement operations extend as $H_I^r = H_{I^r}$ and $H^c_I=H_{I^c}$.

\begin{defn} The \emph{uncoloring} map $\upsilon : \NSym_A \rightarrow \NSym$ is defined by $\upsilon (H_{I}) = H_{w\ell(I)}$ and extended linearly. If the alphabet $A$ only contains one color, then $\upsilon $ is an isomorphism.
\end{defn}

For more details, see \cite{Dau23, doliwa21}.

\subsubsection{The colored quasisymmetric functions and duality} \label{QsymA}

The dual algebra to 
$NSym_A$ is called $QSym_A$, and is constructed using partially commutative colored variables. For a color $a \in A$, consider the set of infinite colored variables $x_a = \{x_{a,1}, x_{a,2}, \ldots \}$ and let $\displaystyle x_A = \bigcup_{a \in A} x_a$. These variables are partially commutative in the sense that they only commute if the second indices are different.  That is, for $a,b \in A$,
$$x_{a,i}x_{b,j}=x_{b,j}x_{a,i} \text{ for } i \not= j \qquad \text{and}\qquad x_{a,i}x_{b,i} \not= x_{b,i}x_{a,i} \text{ if } a \not= b.$$  It follows that every monomial in variables $x_A$ can be uniquely re-ordered so that its second indices are weakly increasing and any first indices sharing the same second index can be combined into a single word. Each monomial can be assocaited with a sentence $(w_1, \ldots, w_m)$ defined by its re-ordered and combined form $x_{w_1,j_1}\cdots x_{w_m, j_m}$ where $j_1 < \ldots < j_m$. Similar notions of coloring, albeit with different assumptions of partial commutativity, can be found in \cite{other_color, poirier}.

\begin{ex}
The monomial $x_{c,2}x_{a,3}x_{b,1}x_{a,2}$ is reordered as $ x_{b,1}x_{c,2}x_{a,2}x_{a,3}$ and combined as $x_{b,1}x_{ca,2}x_{a,3}$.  Then, the sentence associated to this monomial is $(b, ca, a)$. 
\end{ex}

$QSym_A$ is defined as the set of formal power series in $\mathbb{Q}[x_A]$ such that the coefficients of monomials associated with the same sentence are equal.

\begin{ex} The following function $f(x_A)$ is in $QSym_A$:
$$ f(x_A) = 3x_{a,1}x_{bc,2} + 3x_{a,1}x_{bc,3} + \ldots + 3x_{a,2}x_{bc,3} + 3x_{a,2}x_{bc,4} + \ldots .$$
\end{ex}

 For a sentence $I = (w_1, w_2, \ldots , w_m)$, the \emph{colored monomial quasisymmetric function} $M_I$ is defined as 
$$M_I = \sum_{1 \leq j_1<j_2<\ldots <j_m}x_{w_1,j_1}x_{w_2,j_2}\ldots x_{w_m, j_m},$$ where the sum runs over strictly increasing sequences of $m$ positive integers $j_1, \ldots, j_m \in \mathbb{Z}_{>0}$. This basis naturally extends the monomial basis of $QSym$ and is dual to the $\{H_I\}_I$ basis in $\NSym_A$.

\begin{ex} The colored monomial quasisymmetric function for the sentence $(ab,c)$ is
 $$M_{(ab,c)} = x_{ab,1}x_{c,2} + x_{ab,1}x_{c,3} + \ldots  + x_{ab,2}x_{c,3} + x_{ab,2}x_{c,4} + \ldots  + x_{ab,3}x_{c,4} + \ldots.$$ 
\end{ex}

\begin{prop} \cite{doliwa21} $QSym_A$ and $\NSym_A$ are dual Hopf algebras with the inner product $\langle H_I, M_J \rangle = \delta_{I,J}$.
\end{prop}

Multiplication in $QSym_A$ and comultiplication in $\NSym_A$ can be expressed using a variant of shuffling. The quasishuffle $I \Qshuffle J$ is defined as the sum of all shuffles of sentences $I$ and $J$ and of all shuffles of sentences $I$ and $J$ with any number of pairs $w_iv_j$ of consecutive words $w_i \in I$ and $v_j \in J$ concatenated.

\begin{ex} The usual shuffle operation on $(a,bc)$ and $(d,e)$ is
$$(a,bc) \shuffle (d,e) = (a,bc,d,e) + (a,d,bc,e) + (d,a,bc,e) + (a,d,e,bc) + (d,a,e,bc) + (d,e,a,bc).$$
The quasishuffle of $(a,bc)$ and $(d,e)$ is
\begin{multline*}
    (a,bc) \Qshuffle (d,e) = (a,bc,d,e) + (a,bcd,e)+(a,d,bc,e)+ (ad,bc,e)+ (a,d,bce)+ (ad,bce)+ \\ +(d,a,bc,e) + (d,a,bce) + (a,d,e,bc) + (ad,e,bc) + (d,a,e,bc) + (d,ae,bc) + (d,e,a,bc).
\end{multline*}
\end{ex} 

\noindent  Multiplication in $QSym_A$ is dual to the comultiplication $\Delta$ in $\NSym_A$, thus it is given by $$M_IM_J = \sum_K M_K,$$ where  the sum runs over all $K$ in $I \Qshuffle J$. Similaryl, comultiplication in $QSym_A$ is defined by multiplication in $\NSym_A$,
\begin{equation*}
    \Delta(M_I) = \sum_{I=J\cdot K}M_J \otimes M_K,
\end{equation*}
where the sum runs over all sentences $J$ and $K$ such that $I = J \cdot K$. Finally, the antipode $S^*$ in $QSym_A$ is given by $$S^*(M_I) = (-1)^{\ell(I)}\sum_{J^r \succeq I}M_J.$$ 

\begin{defn} The \emph{uncoloring} map $\upsilon : QSym_A \rightarrow QSym$ is defined by $\upsilon (x_{w_1,1}\cdots x_{w_k,k}) = x_1^{|w_1|}\cdots x_k^{|w_k|}$ and extends linearly. If the alphabet $A$ contains only one color, $\upsilon $ is an isomorphism.
\end{defn}

Note that we use $\upsilon $ to denote the uncoloring maps on both $QSym_A$ and $\NSym_A$, and often refer to these together as if they are one map.

\section{The algebra of p-sentences: $PSym_A$}\label{sec:psyma} 

Let $A$ be an alphabet with a total order. Define the \emph{graded lexicographic order} on words by $v \preceq_{g\ell} w$ if $\ell(v) < \ell(w)$ or if both $\ell(v)=\ell(w)$ and $v \preceq_{\ell} w$. A \emph{\psent}is a sentence $P = (w_1, \ldots, w_k)$ such that $w_1 \succeq_{g\ell} \cdots \succeq_{g\ell} w_k$, in other words, a sorted sentence. Given a sentence $I = (v_1, \ldots, v_j)$, let $sort(I)$ be the \psent obtained by sorting the words in $I$ into graded lexicographic order. 
Given a weak sentence $K = (u_1, \ldots, u_m)$, let $sort(K) = sort(\tilde{K})$ and let $\sigma(K)$ be $(u_{\sigma(1)}, \ldots, u_{\sigma(m)})$ for $\sigma \in S_m$. 
Let $PSent_A$ denote the set of \psents for the alphabet $A$. Let $WSent_A$ denote the set of weak sentences and $Sent_A$ denote the set of sentences. 
\begin{ex}
  The sentence $P = (abb,cab,ba,cc,a,b)$ is a p-sentence. Given the sentence $I = (c,aba,bc)$, the associated p-sentence is $sort(I)=(aba,bc,c)$. Given the permutation $\sigma = 231$, we have $\sigma(I) = (bc,c,aba)$.
\end{ex}

\begin{defn}
Define the algebra $PSym_A$ as the algebra generated by $\{h_w\}_w$ for all words $w$ in the alphabet $A$ where generators commute, that is $h_wh_v = h_vh_w$.   We write $$PSym_A = \mathbb{Q}[h_w : \text{words } w].$$  
For a \psent $P = (w_1, \ldots, w_k)$, we define the \emph{colored complete homogeneous symmetric function} as $$h_P = h_{w_1} h_{w_2} \ldots h_{w_k}.$$
\end{defn}

We have constructed this algebra so that the set of colored complete homogeneous $A$-symmetric functions is a basis of $PSym_A$. Using this basis, we show that $PSym_A$ admits a Hopf algebra structure.

\begin{thm}
    $PSym_A$ is a graded Hopf algebra with multiplication given by
    $$h_P h_Q = h_{sort(P \cdot Q)},$$
    the natural unity map $u(k) = k \cdot 1$, the comultiplication given by
    $$\Delta(h_Q) = \sum_{J \subseteq_R Q} h_{sort(Q/_R J)} \otimes h_{sort(J)},$$ and the counit $$\epsilon(h_P) = \begin{cases} 1 & \text{ if } P = \emptyset,\\
0 & \text{ otherwise.}
\end{cases}.$$ 
\end{thm}

\begin{proof}
    We have already defined $PSym_A$ as an algebra, so we begin by showing that $(PSym_A, \Delta, \epsilon)$ is a coalgebra. Observe that
    \begin{align}\label{sumline}
         (id \otimes \Delta) \circ \Delta(h_Q) &= (id \otimes \Delta)\left(\sum_{J \subseteq_R Q} h_{sort(Q/_R J)} \otimes h_{sort(J)}\right) \nonumber\\ 
         &=\sum_{J \subseteq_R Q} h_{sort(Q/_R J)} \otimes \left[ \sum_{K \subseteq_R sort(J)} h_{sort(sort(J)/_R K)} \otimes h_{sort(K)} \right] \nonumber\\ 
         &= \sum_{J \subseteq_R Q,} \sum_{K \subseteq_R J} h_{sort(Q/_R J)} \otimes  h_{sort(J/_R K)} \otimes h_{sort(K)}.
         \end{align}
         To prove co-associativity, we show that $(\Delta \otimes id) \circ \Delta(h_P)$ is equal to \eqref{sumline}.
         \begin{align}\label{sumline2}
        (\Delta \otimes id  ) \circ \Delta(h_Q) &= (\Delta \otimes id)\left(\sum_{Y \subseteq_R Q} h_{sort(Q/_R Y)} \otimes h_{sort(Y)}\right) \nonumber\\
        &= \sum_{Y \subseteq_R Q,} \sum_{Z \subseteq_R sort(Q/_R Y)} h_{sort(sort(Q/_R Y)/_R Z)} \otimes h_{sort(Z)} \otimes h_{sort(Y)}.
        \end{align}

We now reorganize this double sum so that the sum over $Y$ is on the inside and we have a new sum on the outside. Each term in Equation \eqref{sumline2} sum is associated with a unique pair of weak sentences $Y$ and $Z$ such that $Y \subseteq_R Q$ and $Z \subseteq_R sort(Q/_R Y)$.  
Each pair $Y$ and $Z$ can be associated with a unique pair of weak sentences $X$ and $Y$ as follows. 
For each $Y \subseteq_R Q$, pick a permutation $\sigma_Y$ such that $\sigma_Y(Q/_R Y) = sort(Q/_R Y)$, and note that $\sigma_Y(Q/_R Y) = \sigma_Y(Q) /_R \sigma_Y(Y)$. 
Then for a pair $(Y,Z)$ we can define a weak sentence $X$ such that $sort(Q /_R Y) /_R Z =( \sigma_Y(Q) /_R \sigma_Y(Y)) /_R Z = \sigma_Y(Q) /_R \sigma_Y(X)$. In other words, $\sigma_Y(X) /_R \sigma_Y(Y) = Z$. 

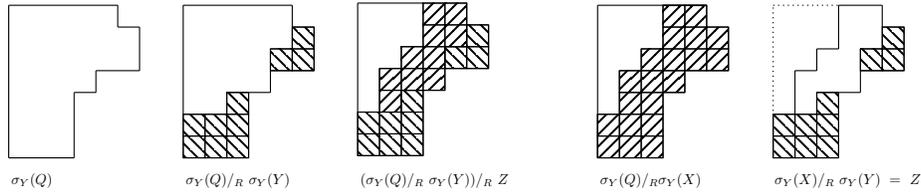
\begin{figure}[h!]
    \centering

 
\tikzset{
pattern size/.store in=\mcSize, 
pattern size = 5pt,
pattern thickness/.store in=\mcThickness, 
pattern thickness = 0.3pt,
pattern radius/.store in=\mcRadius, 
pattern radius = 1pt}
\makeatletter
\pgfutil@ifundefined{pgf@pattern@name@_a6umk1f39}{
\pgfdeclarepatternformonly[\mcThickness,\mcSize]{_a6umk1f39}
{\pgfqpoint{0pt}{-\mcThickness}}
{\pgfpoint{\mcSize}{\mcSize}}
{\pgfpoint{\mcSize}{\mcSize}}
{
\pgfsetcolor{\tikz@pattern@color}
\pgfsetlinewidth{\mcThickness}
\pgfpathmoveto{\pgfqpoint{0pt}{\mcSize}}
\pgfpathlineto{\pgfpoint{\mcSize+\mcThickness}{-\mcThickness}}
\pgfusepath{stroke}
}}
\makeatother

 
\tikzset{
pattern size/.store in=\mcSize, 
pattern size = 5pt,
pattern thickness/.store in=\mcThickness, 
pattern thickness = 0.3pt,
pattern radius/.store in=\mcRadius, 
pattern radius = 1pt}
\makeatletter
\pgfutil@ifundefined{pgf@pattern@name@_c6rl6fams}{
\pgfdeclarepatternformonly[\mcThickness,\mcSize]{_c6rl6fams}
{\pgfqpoint{0pt}{-\mcThickness}}
{\pgfpoint{\mcSize}{\mcSize}}
{\pgfpoint{\mcSize}{\mcSize}}
{
\pgfsetcolor{\tikz@pattern@color}
\pgfsetlinewidth{\mcThickness}
\pgfpathmoveto{\pgfqpoint{0pt}{\mcSize}}
\pgfpathlineto{\pgfpoint{\mcSize+\mcThickness}{-\mcThickness}}
\pgfusepath{stroke}
}}
\makeatother

 
\tikzset{
pattern size/.store in=\mcSize, 
pattern size = 5pt,
pattern thickness/.store in=\mcThickness, 
pattern thickness = 0.3pt,
pattern radius/.store in=\mcRadius, 
pattern radius = 1pt}
\makeatletter
\pgfutil@ifundefined{pgf@pattern@name@_m3vqgrjl0}{
\pgfdeclarepatternformonly[\mcThickness,\mcSize]{_m3vqgrjl0}
{\pgfqpoint{0pt}{-\mcThickness}}
{\pgfpoint{\mcSize}{\mcSize}}
{\pgfpoint{\mcSize}{\mcSize}}
{
\pgfsetcolor{\tikz@pattern@color}
\pgfsetlinewidth{\mcThickness}
\pgfpathmoveto{\pgfqpoint{0pt}{\mcSize}}
\pgfpathlineto{\pgfpoint{\mcSize+\mcThickness}{-\mcThickness}}
\pgfusepath{stroke}
}}
\makeatother

 
\tikzset{
pattern size/.store in=\mcSize, 
pattern size = 5pt,
pattern thickness/.store in=\mcThickness, 
pattern thickness = 0.3pt,
pattern radius/.store in=\mcRadius, 
pattern radius = 1pt}
\makeatletter
\pgfutil@ifundefined{pgf@pattern@name@_pgcud395u}{
\pgfdeclarepatternformonly[\mcThickness,\mcSize]{_pgcud395u}
{\pgfqpoint{0pt}{-\mcThickness}}
{\pgfpoint{\mcSize}{\mcSize}}
{\pgfpoint{\mcSize}{\mcSize}}
{
\pgfsetcolor{\tikz@pattern@color}
\pgfsetlinewidth{\mcThickness}
\pgfpathmoveto{\pgfqpoint{0pt}{\mcSize}}
\pgfpathlineto{\pgfpoint{\mcSize+\mcThickness}{-\mcThickness}}
\pgfusepath{stroke}
}}
\makeatother

 
\tikzset{
pattern size/.store in=\mcSize, 
pattern size = 5pt,
pattern thickness/.store in=\mcThickness, 
pattern thickness = 0.3pt,
pattern radius/.store in=\mcRadius, 
pattern radius = 1pt}
\makeatletter
\pgfutil@ifundefined{pgf@pattern@name@_3xknzeb79}{
\pgfdeclarepatternformonly[\mcThickness,\mcSize]{_3xknzeb79}
{\pgfqpoint{0pt}{-\mcThickness}}
{\pgfpoint{\mcSize}{\mcSize}}
{\pgfpoint{\mcSize}{\mcSize}}
{
\pgfsetcolor{\tikz@pattern@color}
\pgfsetlinewidth{\mcThickness}
\pgfpathmoveto{\pgfqpoint{0pt}{\mcSize}}
\pgfpathlineto{\pgfpoint{\mcSize+\mcThickness}{-\mcThickness}}
\pgfusepath{stroke}
}}
\makeatother

 
\tikzset{
pattern size/.store in=\mcSize, 
pattern size = 5pt,
pattern thickness/.store in=\mcThickness, 
pattern thickness = 0.3pt,
pattern radius/.store in=\mcRadius, 
pattern radius = 1pt}
\makeatletter
\pgfutil@ifundefined{pgf@pattern@name@_ya59ch2pd}{
\pgfdeclarepatternformonly[\mcThickness,\mcSize]{_ya59ch2pd}
{\pgfqpoint{0pt}{-\mcThickness}}
{\pgfpoint{\mcSize}{\mcSize}}
{\pgfpoint{\mcSize}{\mcSize}}
{
\pgfsetcolor{\tikz@pattern@color}
\pgfsetlinewidth{\mcThickness}
\pgfpathmoveto{\pgfqpoint{0pt}{\mcSize}}
\pgfpathlineto{\pgfpoint{\mcSize+\mcThickness}{-\mcThickness}}
\pgfusepath{stroke}
}}
\makeatother

 
\tikzset{
pattern size/.store in=\mcSize, 
pattern size = 5pt,
pattern thickness/.store in=\mcThickness, 
pattern thickness = 0.3pt,
pattern radius/.store in=\mcRadius, 
pattern radius = 1pt}
\makeatletter
\pgfutil@ifundefined{pgf@pattern@name@_6a2p26zyg}{
\pgfdeclarepatternformonly[\mcThickness,\mcSize]{_6a2p26zyg}
{\pgfqpoint{0pt}{-\mcThickness}}
{\pgfpoint{\mcSize}{\mcSize}}
{\pgfpoint{\mcSize}{\mcSize}}
{
\pgfsetcolor{\tikz@pattern@color}
\pgfsetlinewidth{\mcThickness}
\pgfpathmoveto{\pgfqpoint{0pt}{\mcSize}}
\pgfpathlineto{\pgfpoint{\mcSize+\mcThickness}{-\mcThickness}}
\pgfusepath{stroke}
}}
\makeatother

 
\tikzset{
pattern size/.store in=\mcSize, 
pattern size = 5pt,
pattern thickness/.store in=\mcThickness, 
pattern thickness = 0.3pt,
pattern radius/.store in=\mcRadius, 
pattern radius = 1pt}
\makeatletter
\pgfutil@ifundefined{pgf@pattern@name@_io2cnci7d}{
\pgfdeclarepatternformonly[\mcThickness,\mcSize]{_io2cnci7d}
{\pgfqpoint{0pt}{-\mcThickness}}
{\pgfpoint{\mcSize}{\mcSize}}
{\pgfpoint{\mcSize}{\mcSize}}
{
\pgfsetcolor{\tikz@pattern@color}
\pgfsetlinewidth{\mcThickness}
\pgfpathmoveto{\pgfqpoint{0pt}{\mcSize}}
\pgfpathlineto{\pgfpoint{\mcSize+\mcThickness}{-\mcThickness}}
\pgfusepath{stroke}
}}
\makeatother

 
\tikzset{
pattern size/.store in=\mcSize, 
pattern size = 5pt,
pattern thickness/.store in=\mcThickness, 
pattern thickness = 0.3pt,
pattern radius/.store in=\mcRadius, 
pattern radius = 1pt}
\makeatletter
\pgfutil@ifundefined{pgf@pattern@name@_u2a1anb7z}{
\pgfdeclarepatternformonly[\mcThickness,\mcSize]{_u2a1anb7z}
{\pgfqpoint{0pt}{-\mcThickness}}
{\pgfpoint{\mcSize}{\mcSize}}
{\pgfpoint{\mcSize}{\mcSize}}
{
\pgfsetcolor{\tikz@pattern@color}
\pgfsetlinewidth{\mcThickness}
\pgfpathmoveto{\pgfqpoint{0pt}{\mcSize}}
\pgfpathlineto{\pgfpoint{\mcSize+\mcThickness}{-\mcThickness}}
\pgfusepath{stroke}
}}
\makeatother

 
\tikzset{
pattern size/.store in=\mcSize, 
pattern size = 5pt,
pattern thickness/.store in=\mcThickness, 
pattern thickness = 0.3pt,
pattern radius/.store in=\mcRadius, 
pattern radius = 1pt}
\makeatletter
\pgfutil@ifundefined{pgf@pattern@name@_xe7k92cue}{
\pgfdeclarepatternformonly[\mcThickness,\mcSize]{_xe7k92cue}
{\pgfqpoint{0pt}{-\mcThickness}}
{\pgfpoint{\mcSize}{\mcSize}}
{\pgfpoint{\mcSize}{\mcSize}}
{
\pgfsetcolor{\tikz@pattern@color}
\pgfsetlinewidth{\mcThickness}
\pgfpathmoveto{\pgfqpoint{0pt}{\mcSize}}
\pgfpathlineto{\pgfpoint{\mcSize+\mcThickness}{-\mcThickness}}
\pgfusepath{stroke}
}}
\makeatother

 
\tikzset{
pattern size/.store in=\mcSize, 
pattern size = 5pt,
pattern thickness/.store in=\mcThickness, 
pattern thickness = 0.3pt,
pattern radius/.store in=\mcRadius, 
pattern radius = 1pt}
\makeatletter
\pgfutil@ifundefined{pgf@pattern@name@_yfkron34f}{
\pgfdeclarepatternformonly[\mcThickness,\mcSize]{_yfkron34f}
{\pgfqpoint{0pt}{-\mcThickness}}
{\pgfpoint{\mcSize}{\mcSize}}
{\pgfpoint{\mcSize}{\mcSize}}
{
\pgfsetcolor{\tikz@pattern@color}
\pgfsetlinewidth{\mcThickness}
\pgfpathmoveto{\pgfqpoint{0pt}{\mcSize}}
\pgfpathlineto{\pgfpoint{\mcSize+\mcThickness}{-\mcThickness}}
\pgfusepath{stroke}
}}
\makeatother

 
\tikzset{
pattern size/.store in=\mcSize, 
pattern size = 5pt,
pattern thickness/.store in=\mcThickness, 
pattern thickness = 0.3pt,
pattern radius/.store in=\mcRadius, 
pattern radius = 1pt}
\makeatletter
\pgfutil@ifundefined{pgf@pattern@name@_njhkjo5uj}{
\pgfdeclarepatternformonly[\mcThickness,\mcSize]{_njhkjo5uj}
{\pgfqpoint{0pt}{-\mcThickness}}
{\pgfpoint{\mcSize}{\mcSize}}
{\pgfpoint{\mcSize}{\mcSize}}
{
\pgfsetcolor{\tikz@pattern@color}
\pgfsetlinewidth{\mcThickness}
\pgfpathmoveto{\pgfqpoint{0pt}{\mcSize}}
\pgfpathlineto{\pgfpoint{\mcSize+\mcThickness}{-\mcThickness}}
\pgfusepath{stroke}
}}
\makeatother

 
\tikzset{
pattern size/.store in=\mcSize, 
pattern size = 5pt,
pattern thickness/.store in=\mcThickness, 
pattern thickness = 0.3pt,
pattern radius/.store in=\mcRadius, 
pattern radius = 1pt}
\makeatletter
\pgfutil@ifundefined{pgf@pattern@name@_1p2zk7p6h}{
\pgfdeclarepatternformonly[\mcThickness,\mcSize]{_1p2zk7p6h}
{\pgfqpoint{0pt}{-\mcThickness}}
{\pgfpoint{\mcSize}{\mcSize}}
{\pgfpoint{\mcSize}{\mcSize}}
{
\pgfsetcolor{\tikz@pattern@color}
\pgfsetlinewidth{\mcThickness}
\pgfpathmoveto{\pgfqpoint{0pt}{\mcSize}}
\pgfpathlineto{\pgfpoint{\mcSize+\mcThickness}{-\mcThickness}}
\pgfusepath{stroke}
}}
\makeatother

 
\tikzset{
pattern size/.store in=\mcSize, 
pattern size = 5pt,
pattern thickness/.store in=\mcThickness, 
pattern thickness = 0.3pt,
pattern radius/.store in=\mcRadius, 
pattern radius = 1pt}
\makeatletter
\pgfutil@ifundefined{pgf@pattern@name@_upyub9hd4}{
\pgfdeclarepatternformonly[\mcThickness,\mcSize]{_upyub9hd4}
{\pgfqpoint{0pt}{-\mcThickness}}
{\pgfpoint{\mcSize}{\mcSize}}
{\pgfpoint{\mcSize}{\mcSize}}
{
\pgfsetcolor{\tikz@pattern@color}
\pgfsetlinewidth{\mcThickness}
\pgfpathmoveto{\pgfqpoint{0pt}{\mcSize}}
\pgfpathlineto{\pgfpoint{\mcSize+\mcThickness}{-\mcThickness}}
\pgfusepath{stroke}
}}
\makeatother

 
\tikzset{
pattern size/.store in=\mcSize, 
pattern size = 5pt,
pattern thickness/.store in=\mcThickness, 
pattern thickness = 0.3pt,
pattern radius/.store in=\mcRadius, 
pattern radius = 1pt}
\makeatletter
\pgfutil@ifundefined{pgf@pattern@name@_h7get6sa2}{
\pgfdeclarepatternformonly[\mcThickness,\mcSize]{_h7get6sa2}
{\pgfqpoint{0pt}{-\mcThickness}}
{\pgfpoint{\mcSize}{\mcSize}}
{\pgfpoint{\mcSize}{\mcSize}}
{
\pgfsetcolor{\tikz@pattern@color}
\pgfsetlinewidth{\mcThickness}
\pgfpathmoveto{\pgfqpoint{0pt}{\mcSize}}
\pgfpathlineto{\pgfpoint{\mcSize+\mcThickness}{-\mcThickness}}
\pgfusepath{stroke}
}}
\makeatother

 
\tikzset{
pattern size/.store in=\mcSize, 
pattern size = 5pt,
pattern thickness/.store in=\mcThickness, 
pattern thickness = 0.3pt,
pattern radius/.store in=\mcRadius, 
pattern radius = 1pt}
\makeatletter
\pgfutil@ifundefined{pgf@pattern@name@_c793v0goe}{
\pgfdeclarepatternformonly[\mcThickness,\mcSize]{_c793v0goe}
{\pgfqpoint{0pt}{-\mcThickness}}
{\pgfpoint{\mcSize}{\mcSize}}
{\pgfpoint{\mcSize}{\mcSize}}
{
\pgfsetcolor{\tikz@pattern@color}
\pgfsetlinewidth{\mcThickness}
\pgfpathmoveto{\pgfqpoint{0pt}{\mcSize}}
\pgfpathlineto{\pgfpoint{\mcSize+\mcThickness}{-\mcThickness}}
\pgfusepath{stroke}
}}
\makeatother

 
\tikzset{
pattern size/.store in=\mcSize, 
pattern size = 5pt,
pattern thickness/.store in=\mcThickness, 
pattern thickness = 0.3pt,
pattern radius/.store in=\mcRadius, 
pattern radius = 1pt}
\makeatletter
\pgfutil@ifundefined{pgf@pattern@name@_dkmkhz45c}{
\pgfdeclarepatternformonly[\mcThickness,\mcSize]{_dkmkhz45c}
{\pgfqpoint{0pt}{-\mcThickness}}
{\pgfpoint{\mcSize}{\mcSize}}
{\pgfpoint{\mcSize}{\mcSize}}
{
\pgfsetcolor{\tikz@pattern@color}
\pgfsetlinewidth{\mcThickness}
\pgfpathmoveto{\pgfqpoint{0pt}{\mcSize}}
\pgfpathlineto{\pgfpoint{\mcSize+\mcThickness}{-\mcThickness}}
\pgfusepath{stroke}
}}
\makeatother

 
\tikzset{
pattern size/.store in=\mcSize, 
pattern size = 5pt,
pattern thickness/.store in=\mcThickness, 
pattern thickness = 0.3pt,
pattern radius/.store in=\mcRadius, 
pattern radius = 1pt}
\makeatletter
\pgfutil@ifundefined{pgf@pattern@name@_uhj0vhts6}{
\pgfdeclarepatternformonly[\mcThickness,\mcSize]{_uhj0vhts6}
{\pgfqpoint{0pt}{-\mcThickness}}
{\pgfpoint{\mcSize}{\mcSize}}
{\pgfpoint{\mcSize}{\mcSize}}
{
\pgfsetcolor{\tikz@pattern@color}
\pgfsetlinewidth{\mcThickness}
\pgfpathmoveto{\pgfqpoint{0pt}{\mcSize}}
\pgfpathlineto{\pgfpoint{\mcSize+\mcThickness}{-\mcThickness}}
\pgfusepath{stroke}
}}
\makeatother

 
\tikzset{
pattern size/.store in=\mcSize, 
pattern size = 5pt,
pattern thickness/.store in=\mcThickness, 
pattern thickness = 0.3pt,
pattern radius/.store in=\mcRadius, 
pattern radius = 1pt}
\makeatletter
\pgfutil@ifundefined{pgf@pattern@name@_6g29qa4yr}{
\pgfdeclarepatternformonly[\mcThickness,\mcSize]{_6g29qa4yr}
{\pgfqpoint{0pt}{-\mcThickness}}
{\pgfpoint{\mcSize}{\mcSize}}
{\pgfpoint{\mcSize}{\mcSize}}
{
\pgfsetcolor{\tikz@pattern@color}
\pgfsetlinewidth{\mcThickness}
\pgfpathmoveto{\pgfqpoint{0pt}{\mcSize}}
\pgfpathlineto{\pgfpoint{\mcSize+\mcThickness}{-\mcThickness}}
\pgfusepath{stroke}
}}
\makeatother

 
\tikzset{
pattern size/.store in=\mcSize, 
pattern size = 5pt,
pattern thickness/.store in=\mcThickness, 
pattern thickness = 0.3pt,
pattern radius/.store in=\mcRadius, 
pattern radius = 1pt}
\makeatletter
\pgfutil@ifundefined{pgf@pattern@name@_ba3143j47}{
\pgfdeclarepatternformonly[\mcThickness,\mcSize]{_ba3143j47}
{\pgfqpoint{0pt}{-\mcThickness}}
{\pgfpoint{\mcSize}{\mcSize}}
{\pgfpoint{\mcSize}{\mcSize}}
{
\pgfsetcolor{\tikz@pattern@color}
\pgfsetlinewidth{\mcThickness}
\pgfpathmoveto{\pgfqpoint{0pt}{\mcSize}}
\pgfpathlineto{\pgfpoint{\mcSize+\mcThickness}{-\mcThickness}}
\pgfusepath{stroke}
}}
\makeatother

 
\tikzset{
pattern size/.store in=\mcSize, 
pattern size = 5pt,
pattern thickness/.store in=\mcThickness, 
pattern thickness = 0.3pt,
pattern radius/.store in=\mcRadius, 
pattern radius = 1pt}
\makeatletter
\pgfutil@ifundefined{pgf@pattern@name@_0knsan5ho}{
\pgfdeclarepatternformonly[\mcThickness,\mcSize]{_0knsan5ho}
{\pgfqpoint{0pt}{-\mcThickness}}
{\pgfpoint{\mcSize}{\mcSize}}
{\pgfpoint{\mcSize}{\mcSize}}
{
\pgfsetcolor{\tikz@pattern@color}
\pgfsetlinewidth{\mcThickness}
\pgfpathmoveto{\pgfqpoint{0pt}{\mcSize}}
\pgfpathlineto{\pgfpoint{\mcSize+\mcThickness}{-\mcThickness}}
\pgfusepath{stroke}
}}
\makeatother

 
\tikzset{
pattern size/.store in=\mcSize, 
pattern size = 5pt,
pattern thickness/.store in=\mcThickness, 
pattern thickness = 0.3pt,
pattern radius/.store in=\mcRadius, 
pattern radius = 1pt}
\makeatletter
\pgfutil@ifundefined{pgf@pattern@name@_d2l82x0zk}{
\pgfdeclarepatternformonly[\mcThickness,\mcSize]{_d2l82x0zk}
{\pgfqpoint{0pt}{0pt}}
{\pgfpoint{\mcSize+\mcThickness}{\mcSize+\mcThickness}}
{\pgfpoint{\mcSize}{\mcSize}}
{
\pgfsetcolor{\tikz@pattern@color}
\pgfsetlinewidth{\mcThickness}
\pgfpathmoveto{\pgfqpoint{0pt}{0pt}}
\pgfpathlineto{\pgfpoint{\mcSize+\mcThickness}{\mcSize+\mcThickness}}
\pgfusepath{stroke}
}}
\makeatother

 
\tikzset{
pattern size/.store in=\mcSize, 
pattern size = 5pt,
pattern thickness/.store in=\mcThickness, 
pattern thickness = 0.3pt,
pattern radius/.store in=\mcRadius, 
pattern radius = 1pt}
\makeatletter
\pgfutil@ifundefined{pgf@pattern@name@_ylm12zkue}{
\pgfdeclarepatternformonly[\mcThickness,\mcSize]{_ylm12zkue}
{\pgfqpoint{0pt}{0pt}}
{\pgfpoint{\mcSize+\mcThickness}{\mcSize+\mcThickness}}
{\pgfpoint{\mcSize}{\mcSize}}
{
\pgfsetcolor{\tikz@pattern@color}
\pgfsetlinewidth{\mcThickness}
\pgfpathmoveto{\pgfqpoint{0pt}{0pt}}
\pgfpathlineto{\pgfpoint{\mcSize+\mcThickness}{\mcSize+\mcThickness}}
\pgfusepath{stroke}
}}
\makeatother

 
\tikzset{
pattern size/.store in=\mcSize, 
pattern size = 5pt,
pattern thickness/.store in=\mcThickness, 
pattern thickness = 0.3pt,
pattern radius/.store in=\mcRadius, 
pattern radius = 1pt}
\makeatletter
\pgfutil@ifundefined{pgf@pattern@name@_wcj7owzr7}{
\pgfdeclarepatternformonly[\mcThickness,\mcSize]{_wcj7owzr7}
{\pgfqpoint{0pt}{0pt}}
{\pgfpoint{\mcSize+\mcThickness}{\mcSize+\mcThickness}}
{\pgfpoint{\mcSize}{\mcSize}}
{
\pgfsetcolor{\tikz@pattern@color}
\pgfsetlinewidth{\mcThickness}
\pgfpathmoveto{\pgfqpoint{0pt}{0pt}}
\pgfpathlineto{\pgfpoint{\mcSize+\mcThickness}{\mcSize+\mcThickness}}
\pgfusepath{stroke}
}}
\makeatother

 
\tikzset{
pattern size/.store in=\mcSize, 
pattern size = 5pt,
pattern thickness/.store in=\mcThickness, 
pattern thickness = 0.3pt,
pattern radius/.store in=\mcRadius, 
pattern radius = 1pt}
\makeatletter
\pgfutil@ifundefined{pgf@pattern@name@_nmj0zh67s}{
\pgfdeclarepatternformonly[\mcThickness,\mcSize]{_nmj0zh67s}
{\pgfqpoint{0pt}{0pt}}
{\pgfpoint{\mcSize+\mcThickness}{\mcSize+\mcThickness}}
{\pgfpoint{\mcSize}{\mcSize}}
{
\pgfsetcolor{\tikz@pattern@color}
\pgfsetlinewidth{\mcThickness}
\pgfpathmoveto{\pgfqpoint{0pt}{0pt}}
\pgfpathlineto{\pgfpoint{\mcSize+\mcThickness}{\mcSize+\mcThickness}}
\pgfusepath{stroke}
}}
\makeatother

 
\tikzset{
pattern size/.store in=\mcSize, 
pattern size = 5pt,
pattern thickness/.store in=\mcThickness, 
pattern thickness = 0.3pt,
pattern radius/.store in=\mcRadius, 
pattern radius = 1pt}
\makeatletter
\pgfutil@ifundefined{pgf@pattern@name@_vdpemwr0m}{
\pgfdeclarepatternformonly[\mcThickness,\mcSize]{_vdpemwr0m}
{\pgfqpoint{0pt}{0pt}}
{\pgfpoint{\mcSize+\mcThickness}{\mcSize+\mcThickness}}
{\pgfpoint{\mcSize}{\mcSize}}
{
\pgfsetcolor{\tikz@pattern@color}
\pgfsetlinewidth{\mcThickness}
\pgfpathmoveto{\pgfqpoint{0pt}{0pt}}
\pgfpathlineto{\pgfpoint{\mcSize+\mcThickness}{\mcSize+\mcThickness}}
\pgfusepath{stroke}
}}
\makeatother

 
\tikzset{
pattern size/.store in=\mcSize, 
pattern size = 5pt,
pattern thickness/.store in=\mcThickness, 
pattern thickness = 0.3pt,
pattern radius/.store in=\mcRadius, 
pattern radius = 1pt}
\makeatletter
\pgfutil@ifundefined{pgf@pattern@name@_dr2197wcf}{
\pgfdeclarepatternformonly[\mcThickness,\mcSize]{_dr2197wcf}
{\pgfqpoint{0pt}{0pt}}
{\pgfpoint{\mcSize+\mcThickness}{\mcSize+\mcThickness}}
{\pgfpoint{\mcSize}{\mcSize}}
{
\pgfsetcolor{\tikz@pattern@color}
\pgfsetlinewidth{\mcThickness}
\pgfpathmoveto{\pgfqpoint{0pt}{0pt}}
\pgfpathlineto{\pgfpoint{\mcSize+\mcThickness}{\mcSize+\mcThickness}}
\pgfusepath{stroke}
}}
\makeatother

 
\tikzset{
pattern size/.store in=\mcSize, 
pattern size = 5pt,
pattern thickness/.store in=\mcThickness, 
pattern thickness = 0.3pt,
pattern radius/.store in=\mcRadius, 
pattern radius = 1pt}
\makeatletter
\pgfutil@ifundefined{pgf@pattern@name@_3qcqbesgc}{
\pgfdeclarepatternformonly[\mcThickness,\mcSize]{_3qcqbesgc}
{\pgfqpoint{0pt}{0pt}}
{\pgfpoint{\mcSize+\mcThickness}{\mcSize+\mcThickness}}
{\pgfpoint{\mcSize}{\mcSize}}
{
\pgfsetcolor{\tikz@pattern@color}
\pgfsetlinewidth{\mcThickness}
\pgfpathmoveto{\pgfqpoint{0pt}{0pt}}
\pgfpathlineto{\pgfpoint{\mcSize+\mcThickness}{\mcSize+\mcThickness}}
\pgfusepath{stroke}
}}
\makeatother

 
\tikzset{
pattern size/.store in=\mcSize, 
pattern size = 5pt,
pattern thickness/.store in=\mcThickness, 
pattern thickness = 0.3pt,
pattern radius/.store in=\mcRadius, 
pattern radius = 1pt}
\makeatletter
\pgfutil@ifundefined{pgf@pattern@name@_ms4rkwdd5}{
\pgfdeclarepatternformonly[\mcThickness,\mcSize]{_ms4rkwdd5}
{\pgfqpoint{0pt}{0pt}}
{\pgfpoint{\mcSize+\mcThickness}{\mcSize+\mcThickness}}
{\pgfpoint{\mcSize}{\mcSize}}
{
\pgfsetcolor{\tikz@pattern@color}
\pgfsetlinewidth{\mcThickness}
\pgfpathmoveto{\pgfqpoint{0pt}{0pt}}
\pgfpathlineto{\pgfpoint{\mcSize+\mcThickness}{\mcSize+\mcThickness}}
\pgfusepath{stroke}
}}
\makeatother

 
\tikzset{
pattern size/.store in=\mcSize, 
pattern size = 5pt,
pattern thickness/.store in=\mcThickness, 
pattern thickness = 0.3pt,
pattern radius/.store in=\mcRadius, 
pattern radius = 1pt}
\makeatletter
\pgfutil@ifundefined{pgf@pattern@name@_vtybccwm3}{
\pgfdeclarepatternformonly[\mcThickness,\mcSize]{_vtybccwm3}
{\pgfqpoint{0pt}{0pt}}
{\pgfpoint{\mcSize+\mcThickness}{\mcSize+\mcThickness}}
{\pgfpoint{\mcSize}{\mcSize}}
{
\pgfsetcolor{\tikz@pattern@color}
\pgfsetlinewidth{\mcThickness}
\pgfpathmoveto{\pgfqpoint{0pt}{0pt}}
\pgfpathlineto{\pgfpoint{\mcSize+\mcThickness}{\mcSize+\mcThickness}}
\pgfusepath{stroke}
}}
\makeatother

 
\tikzset{
pattern size/.store in=\mcSize, 
pattern size = 5pt,
pattern thickness/.store in=\mcThickness, 
pattern thickness = 0.3pt,
pattern radius/.store in=\mcRadius, 
pattern radius = 1pt}
\makeatletter
\pgfutil@ifundefined{pgf@pattern@name@_hos41ragy}{
\pgfdeclarepatternformonly[\mcThickness,\mcSize]{_hos41ragy}
{\pgfqpoint{0pt}{0pt}}
{\pgfpoint{\mcSize+\mcThickness}{\mcSize+\mcThickness}}
{\pgfpoint{\mcSize}{\mcSize}}
{
\pgfsetcolor{\tikz@pattern@color}
\pgfsetlinewidth{\mcThickness}
\pgfpathmoveto{\pgfqpoint{0pt}{0pt}}
\pgfpathlineto{\pgfpoint{\mcSize+\mcThickness}{\mcSize+\mcThickness}}
\pgfusepath{stroke}
}}
\makeatother

 
\tikzset{
pattern size/.store in=\mcSize, 
pattern size = 5pt,
pattern thickness/.store in=\mcThickness, 
pattern thickness = 0.3pt,
pattern radius/.store in=\mcRadius, 
pattern radius = 1pt}
\makeatletter
\pgfutil@ifundefined{pgf@pattern@name@_0wabbbft6}{
\pgfdeclarepatternformonly[\mcThickness,\mcSize]{_0wabbbft6}
{\pgfqpoint{0pt}{0pt}}
{\pgfpoint{\mcSize+\mcThickness}{\mcSize+\mcThickness}}
{\pgfpoint{\mcSize}{\mcSize}}
{
\pgfsetcolor{\tikz@pattern@color}
\pgfsetlinewidth{\mcThickness}
\pgfpathmoveto{\pgfqpoint{0pt}{0pt}}
\pgfpathlineto{\pgfpoint{\mcSize+\mcThickness}{\mcSize+\mcThickness}}
\pgfusepath{stroke}
}}
\makeatother

 
\tikzset{
pattern size/.store in=\mcSize, 
pattern size = 5pt,
pattern thickness/.store in=\mcThickness, 
pattern thickness = 0.3pt,
pattern radius/.store in=\mcRadius, 
pattern radius = 1pt}
\makeatletter
\pgfutil@ifundefined{pgf@pattern@name@_oqzcc7lxe}{
\pgfdeclarepatternformonly[\mcThickness,\mcSize]{_oqzcc7lxe}
{\pgfqpoint{0pt}{-\mcThickness}}
{\pgfpoint{\mcSize}{\mcSize}}
{\pgfpoint{\mcSize}{\mcSize}}
{
\pgfsetcolor{\tikz@pattern@color}
\pgfsetlinewidth{\mcThickness}
\pgfpathmoveto{\pgfqpoint{0pt}{\mcSize}}
\pgfpathlineto{\pgfpoint{\mcSize+\mcThickness}{-\mcThickness}}
\pgfusepath{stroke}
}}
\makeatother

 
\tikzset{
pattern size/.store in=\mcSize, 
pattern size = 5pt,
pattern thickness/.store in=\mcThickness, 
pattern thickness = 0.3pt,
pattern radius/.store in=\mcRadius, 
pattern radius = 1pt}
\makeatletter
\pgfutil@ifundefined{pgf@pattern@name@_sn8wa2re8}{
\pgfdeclarepatternformonly[\mcThickness,\mcSize]{_sn8wa2re8}
{\pgfqpoint{0pt}{-\mcThickness}}
{\pgfpoint{\mcSize}{\mcSize}}
{\pgfpoint{\mcSize}{\mcSize}}
{
\pgfsetcolor{\tikz@pattern@color}
\pgfsetlinewidth{\mcThickness}
\pgfpathmoveto{\pgfqpoint{0pt}{\mcSize}}
\pgfpathlineto{\pgfpoint{\mcSize+\mcThickness}{-\mcThickness}}
\pgfusepath{stroke}
}}
\makeatother

 
\tikzset{
pattern size/.store in=\mcSize, 
pattern size = 5pt,
pattern thickness/.store in=\mcThickness, 
pattern thickness = 0.3pt,
pattern radius/.store in=\mcRadius, 
pattern radius = 1pt}
\makeatletter
\pgfutil@ifundefined{pgf@pattern@name@_dt5bp7kd7}{
\pgfdeclarepatternformonly[\mcThickness,\mcSize]{_dt5bp7kd7}
{\pgfqpoint{0pt}{-\mcThickness}}
{\pgfpoint{\mcSize}{\mcSize}}
{\pgfpoint{\mcSize}{\mcSize}}
{
\pgfsetcolor{\tikz@pattern@color}
\pgfsetlinewidth{\mcThickness}
\pgfpathmoveto{\pgfqpoint{0pt}{\mcSize}}
\pgfpathlineto{\pgfpoint{\mcSize+\mcThickness}{-\mcThickness}}
\pgfusepath{stroke}
}}
\makeatother

 
\tikzset{
pattern size/.store in=\mcSize, 
pattern size = 5pt,
pattern thickness/.store in=\mcThickness, 
pattern thickness = 0.3pt,
pattern radius/.store in=\mcRadius, 
pattern radius = 1pt}
\makeatletter
\pgfutil@ifundefined{pgf@pattern@name@_m3ss30n5f}{
\pgfdeclarepatternformonly[\mcThickness,\mcSize]{_m3ss30n5f}
{\pgfqpoint{0pt}{-\mcThickness}}
{\pgfpoint{\mcSize}{\mcSize}}
{\pgfpoint{\mcSize}{\mcSize}}
{
\pgfsetcolor{\tikz@pattern@color}
\pgfsetlinewidth{\mcThickness}
\pgfpathmoveto{\pgfqpoint{0pt}{\mcSize}}
\pgfpathlineto{\pgfpoint{\mcSize+\mcThickness}{-\mcThickness}}
\pgfusepath{stroke}
}}
\makeatother

 
\tikzset{
pattern size/.store in=\mcSize, 
pattern size = 5pt,
pattern thickness/.store in=\mcThickness, 
pattern thickness = 0.3pt,
pattern radius/.store in=\mcRadius, 
pattern radius = 1pt}
\makeatletter
\pgfutil@ifundefined{pgf@pattern@name@_eeaq4pblw}{
\pgfdeclarepatternformonly[\mcThickness,\mcSize]{_eeaq4pblw}
{\pgfqpoint{0pt}{-\mcThickness}}
{\pgfpoint{\mcSize}{\mcSize}}
{\pgfpoint{\mcSize}{\mcSize}}
{
\pgfsetcolor{\tikz@pattern@color}
\pgfsetlinewidth{\mcThickness}
\pgfpathmoveto{\pgfqpoint{0pt}{\mcSize}}
\pgfpathlineto{\pgfpoint{\mcSize+\mcThickness}{-\mcThickness}}
\pgfusepath{stroke}
}}
\makeatother

 
\tikzset{
pattern size/.store in=\mcSize, 
pattern size = 5pt,
pattern thickness/.store in=\mcThickness, 
pattern thickness = 0.3pt,
pattern radius/.store in=\mcRadius, 
pattern radius = 1pt}
\makeatletter
\pgfutil@ifundefined{pgf@pattern@name@_buxsekv3w}{
\pgfdeclarepatternformonly[\mcThickness,\mcSize]{_buxsekv3w}
{\pgfqpoint{0pt}{-\mcThickness}}
{\pgfpoint{\mcSize}{\mcSize}}
{\pgfpoint{\mcSize}{\mcSize}}
{
\pgfsetcolor{\tikz@pattern@color}
\pgfsetlinewidth{\mcThickness}
\pgfpathmoveto{\pgfqpoint{0pt}{\mcSize}}
\pgfpathlineto{\pgfpoint{\mcSize+\mcThickness}{-\mcThickness}}
\pgfusepath{stroke}
}}
\makeatother

 
\tikzset{
pattern size/.store in=\mcSize, 
pattern size = 5pt,
pattern thickness/.store in=\mcThickness, 
pattern thickness = 0.3pt,
pattern radius/.store in=\mcRadius, 
pattern radius = 1pt}
\makeatletter
\pgfutil@ifundefined{pgf@pattern@name@_f6ixcbm2e}{
\pgfdeclarepatternformonly[\mcThickness,\mcSize]{_f6ixcbm2e}
{\pgfqpoint{0pt}{-\mcThickness}}
{\pgfpoint{\mcSize}{\mcSize}}
{\pgfpoint{\mcSize}{\mcSize}}
{
\pgfsetcolor{\tikz@pattern@color}
\pgfsetlinewidth{\mcThickness}
\pgfpathmoveto{\pgfqpoint{0pt}{\mcSize}}
\pgfpathlineto{\pgfpoint{\mcSize+\mcThickness}{-\mcThickness}}
\pgfusepath{stroke}
}}
\makeatother

 
\tikzset{
pattern size/.store in=\mcSize, 
pattern size = 5pt,
pattern thickness/.store in=\mcThickness, 
pattern thickness = 0.3pt,
pattern radius/.store in=\mcRadius, 
pattern radius = 1pt}
\makeatletter
\pgfutil@ifundefined{pgf@pattern@name@_eh6fw3syn}{
\pgfdeclarepatternformonly[\mcThickness,\mcSize]{_eh6fw3syn}
{\pgfqpoint{0pt}{-\mcThickness}}
{\pgfpoint{\mcSize}{\mcSize}}
{\pgfpoint{\mcSize}{\mcSize}}
{
\pgfsetcolor{\tikz@pattern@color}
\pgfsetlinewidth{\mcThickness}
\pgfpathmoveto{\pgfqpoint{0pt}{\mcSize}}
\pgfpathlineto{\pgfpoint{\mcSize+\mcThickness}{-\mcThickness}}
\pgfusepath{stroke}
}}
\makeatother

 
\tikzset{
pattern size/.store in=\mcSize, 
pattern size = 5pt,
pattern thickness/.store in=\mcThickness, 
pattern thickness = 0.3pt,
pattern radius/.store in=\mcRadius, 
pattern radius = 1pt}
\makeatletter
\pgfutil@ifundefined{pgf@pattern@name@_d8eo8rhz5}{
\pgfdeclarepatternformonly[\mcThickness,\mcSize]{_d8eo8rhz5}
{\pgfqpoint{0pt}{-\mcThickness}}
{\pgfpoint{\mcSize}{\mcSize}}
{\pgfpoint{\mcSize}{\mcSize}}
{
\pgfsetcolor{\tikz@pattern@color}
\pgfsetlinewidth{\mcThickness}
\pgfpathmoveto{\pgfqpoint{0pt}{\mcSize}}
\pgfpathlineto{\pgfpoint{\mcSize+\mcThickness}{-\mcThickness}}
\pgfusepath{stroke}
}}
\makeatother

 
\tikzset{
pattern size/.store in=\mcSize, 
pattern size = 5pt,
pattern thickness/.store in=\mcThickness, 
pattern thickness = 0.3pt,
pattern radius/.store in=\mcRadius, 
pattern radius = 1pt}
\makeatletter
\pgfutil@ifundefined{pgf@pattern@name@_dfnugp45i}{
\pgfdeclarepatternformonly[\mcThickness,\mcSize]{_dfnugp45i}
{\pgfqpoint{0pt}{0pt}}
{\pgfpoint{\mcSize+\mcThickness}{\mcSize+\mcThickness}}
{\pgfpoint{\mcSize}{\mcSize}}
{
\pgfsetcolor{\tikz@pattern@color}
\pgfsetlinewidth{\mcThickness}
\pgfpathmoveto{\pgfqpoint{0pt}{0pt}}
\pgfpathlineto{\pgfpoint{\mcSize+\mcThickness}{\mcSize+\mcThickness}}
\pgfusepath{stroke}
}}
\makeatother

 
\tikzset{
pattern size/.store in=\mcSize, 
pattern size = 5pt,
pattern thickness/.store in=\mcThickness, 
pattern thickness = 0.3pt,
pattern radius/.store in=\mcRadius, 
pattern radius = 1pt}
\makeatletter
\pgfutil@ifundefined{pgf@pattern@name@_e8s4tql46}{
\pgfdeclarepatternformonly[\mcThickness,\mcSize]{_e8s4tql46}
{\pgfqpoint{0pt}{0pt}}
{\pgfpoint{\mcSize+\mcThickness}{\mcSize+\mcThickness}}
{\pgfpoint{\mcSize}{\mcSize}}
{
\pgfsetcolor{\tikz@pattern@color}
\pgfsetlinewidth{\mcThickness}
\pgfpathmoveto{\pgfqpoint{0pt}{0pt}}
\pgfpathlineto{\pgfpoint{\mcSize+\mcThickness}{\mcSize+\mcThickness}}
\pgfusepath{stroke}
}}
\makeatother

 
\tikzset{
pattern size/.store in=\mcSize, 
pattern size = 5pt,
pattern thickness/.store in=\mcThickness, 
pattern thickness = 0.3pt,
pattern radius/.store in=\mcRadius, 
pattern radius = 1pt}
\makeatletter
\pgfutil@ifundefined{pgf@pattern@name@_la3bydh2h}{
\pgfdeclarepatternformonly[\mcThickness,\mcSize]{_la3bydh2h}
{\pgfqpoint{0pt}{0pt}}
{\pgfpoint{\mcSize+\mcThickness}{\mcSize+\mcThickness}}
{\pgfpoint{\mcSize}{\mcSize}}
{
\pgfsetcolor{\tikz@pattern@color}
\pgfsetlinewidth{\mcThickness}
\pgfpathmoveto{\pgfqpoint{0pt}{0pt}}
\pgfpathlineto{\pgfpoint{\mcSize+\mcThickness}{\mcSize+\mcThickness}}
\pgfusepath{stroke}
}}
\makeatother

 
\tikzset{
pattern size/.store in=\mcSize, 
pattern size = 5pt,
pattern thickness/.store in=\mcThickness, 
pattern thickness = 0.3pt,
pattern radius/.store in=\mcRadius, 
pattern radius = 1pt}
\makeatletter
\pgfutil@ifundefined{pgf@pattern@name@_wj3mg48y5}{
\pgfdeclarepatternformonly[\mcThickness,\mcSize]{_wj3mg48y5}
{\pgfqpoint{0pt}{0pt}}
{\pgfpoint{\mcSize+\mcThickness}{\mcSize+\mcThickness}}
{\pgfpoint{\mcSize}{\mcSize}}
{
\pgfsetcolor{\tikz@pattern@color}
\pgfsetlinewidth{\mcThickness}
\pgfpathmoveto{\pgfqpoint{0pt}{0pt}}
\pgfpathlineto{\pgfpoint{\mcSize+\mcThickness}{\mcSize+\mcThickness}}
\pgfusepath{stroke}
}}
\makeatother

 
\tikzset{
pattern size/.store in=\mcSize, 
pattern size = 5pt,
pattern thickness/.store in=\mcThickness, 
pattern thickness = 0.3pt,
pattern radius/.store in=\mcRadius, 
pattern radius = 1pt}
\makeatletter
\pgfutil@ifundefined{pgf@pattern@name@_ckpacz1d1}{
\pgfdeclarepatternformonly[\mcThickness,\mcSize]{_ckpacz1d1}
{\pgfqpoint{0pt}{0pt}}
{\pgfpoint{\mcSize+\mcThickness}{\mcSize+\mcThickness}}
{\pgfpoint{\mcSize}{\mcSize}}
{
\pgfsetcolor{\tikz@pattern@color}
\pgfsetlinewidth{\mcThickness}
\pgfpathmoveto{\pgfqpoint{0pt}{0pt}}
\pgfpathlineto{\pgfpoint{\mcSize+\mcThickness}{\mcSize+\mcThickness}}
\pgfusepath{stroke}
}}
\makeatother

 
\tikzset{
pattern size/.store in=\mcSize, 
pattern size = 5pt,
pattern thickness/.store in=\mcThickness, 
pattern thickness = 0.3pt,
pattern radius/.store in=\mcRadius, 
pattern radius = 1pt}
\makeatletter
\pgfutil@ifundefined{pgf@pattern@name@_w1w2jqn62}{
\pgfdeclarepatternformonly[\mcThickness,\mcSize]{_w1w2jqn62}
{\pgfqpoint{0pt}{0pt}}
{\pgfpoint{\mcSize+\mcThickness}{\mcSize+\mcThickness}}
{\pgfpoint{\mcSize}{\mcSize}}
{
\pgfsetcolor{\tikz@pattern@color}
\pgfsetlinewidth{\mcThickness}
\pgfpathmoveto{\pgfqpoint{0pt}{0pt}}
\pgfpathlineto{\pgfpoint{\mcSize+\mcThickness}{\mcSize+\mcThickness}}
\pgfusepath{stroke}
}}
\makeatother

 
\tikzset{
pattern size/.store in=\mcSize, 
pattern size = 5pt,
pattern thickness/.store in=\mcThickness, 
pattern thickness = 0.3pt,
pattern radius/.store in=\mcRadius, 
pattern radius = 1pt}
\makeatletter
\pgfutil@ifundefined{pgf@pattern@name@_crybwhpqx}{
\pgfdeclarepatternformonly[\mcThickness,\mcSize]{_crybwhpqx}
{\pgfqpoint{0pt}{0pt}}
{\pgfpoint{\mcSize+\mcThickness}{\mcSize+\mcThickness}}
{\pgfpoint{\mcSize}{\mcSize}}
{
\pgfsetcolor{\tikz@pattern@color}
\pgfsetlinewidth{\mcThickness}
\pgfpathmoveto{\pgfqpoint{0pt}{0pt}}
\pgfpathlineto{\pgfpoint{\mcSize+\mcThickness}{\mcSize+\mcThickness}}
\pgfusepath{stroke}
}}
\makeatother

 
\tikzset{
pattern size/.store in=\mcSize, 
pattern size = 5pt,
pattern thickness/.store in=\mcThickness, 
pattern thickness = 0.3pt,
pattern radius/.store in=\mcRadius, 
pattern radius = 1pt}
\makeatletter
\pgfutil@ifundefined{pgf@pattern@name@_g0t93lu20}{
\pgfdeclarepatternformonly[\mcThickness,\mcSize]{_g0t93lu20}
{\pgfqpoint{0pt}{0pt}}
{\pgfpoint{\mcSize+\mcThickness}{\mcSize+\mcThickness}}
{\pgfpoint{\mcSize}{\mcSize}}
{
\pgfsetcolor{\tikz@pattern@color}
\pgfsetlinewidth{\mcThickness}
\pgfpathmoveto{\pgfqpoint{0pt}{0pt}}
\pgfpathlineto{\pgfpoint{\mcSize+\mcThickness}{\mcSize+\mcThickness}}
\pgfusepath{stroke}
}}
\makeatother

 
\tikzset{
pattern size/.store in=\mcSize, 
pattern size = 5pt,
pattern thickness/.store in=\mcThickness, 
pattern thickness = 0.3pt,
pattern radius/.store in=\mcRadius, 
pattern radius = 1pt}
\makeatletter
\pgfutil@ifundefined{pgf@pattern@name@_vqxe0agie}{
\pgfdeclarepatternformonly[\mcThickness,\mcSize]{_vqxe0agie}
{\pgfqpoint{0pt}{0pt}}
{\pgfpoint{\mcSize+\mcThickness}{\mcSize+\mcThickness}}
{\pgfpoint{\mcSize}{\mcSize}}
{
\pgfsetcolor{\tikz@pattern@color}
\pgfsetlinewidth{\mcThickness}
\pgfpathmoveto{\pgfqpoint{0pt}{0pt}}
\pgfpathlineto{\pgfpoint{\mcSize+\mcThickness}{\mcSize+\mcThickness}}
\pgfusepath{stroke}
}}
\makeatother

 
\tikzset{
pattern size/.store in=\mcSize, 
pattern size = 5pt,
pattern thickness/.store in=\mcThickness, 
pattern thickness = 0.3pt,
pattern radius/.store in=\mcRadius, 
pattern radius = 1pt}
\makeatletter
\pgfutil@ifundefined{pgf@pattern@name@_8df04g6ll}{
\pgfdeclarepatternformonly[\mcThickness,\mcSize]{_8df04g6ll}
{\pgfqpoint{0pt}{0pt}}
{\pgfpoint{\mcSize+\mcThickness}{\mcSize+\mcThickness}}
{\pgfpoint{\mcSize}{\mcSize}}
{
\pgfsetcolor{\tikz@pattern@color}
\pgfsetlinewidth{\mcThickness}
\pgfpathmoveto{\pgfqpoint{0pt}{0pt}}
\pgfpathlineto{\pgfpoint{\mcSize+\mcThickness}{\mcSize+\mcThickness}}
\pgfusepath{stroke}
}}
\makeatother

 
\tikzset{
pattern size/.store in=\mcSize, 
pattern size = 5pt,
pattern thickness/.store in=\mcThickness, 
pattern thickness = 0.3pt,
pattern radius/.store in=\mcRadius, 
pattern radius = 1pt}
\makeatletter
\pgfutil@ifundefined{pgf@pattern@name@_m4cuigs7k}{
\pgfdeclarepatternformonly[\mcThickness,\mcSize]{_m4cuigs7k}
{\pgfqpoint{0pt}{0pt}}
{\pgfpoint{\mcSize+\mcThickness}{\mcSize+\mcThickness}}
{\pgfpoint{\mcSize}{\mcSize}}
{
\pgfsetcolor{\tikz@pattern@color}
\pgfsetlinewidth{\mcThickness}
\pgfpathmoveto{\pgfqpoint{0pt}{0pt}}
\pgfpathlineto{\pgfpoint{\mcSize+\mcThickness}{\mcSize+\mcThickness}}
\pgfusepath{stroke}
}}
\makeatother

 
\tikzset{
pattern size/.store in=\mcSize, 
pattern size = 5pt,
pattern thickness/.store in=\mcThickness, 
pattern thickness = 0.3pt,
pattern radius/.store in=\mcRadius, 
pattern radius = 1pt}
\makeatletter
\pgfutil@ifundefined{pgf@pattern@name@_x3fbefqtm}{
\pgfdeclarepatternformonly[\mcThickness,\mcSize]{_x3fbefqtm}
{\pgfqpoint{0pt}{0pt}}
{\pgfpoint{\mcSize+\mcThickness}{\mcSize+\mcThickness}}
{\pgfpoint{\mcSize}{\mcSize}}
{
\pgfsetcolor{\tikz@pattern@color}
\pgfsetlinewidth{\mcThickness}
\pgfpathmoveto{\pgfqpoint{0pt}{0pt}}
\pgfpathlineto{\pgfpoint{\mcSize+\mcThickness}{\mcSize+\mcThickness}}
\pgfusepath{stroke}
}}
\makeatother

 
\tikzset{
pattern size/.store in=\mcSize, 
pattern size = 5pt,
pattern thickness/.store in=\mcThickness, 
pattern thickness = 0.3pt,
pattern radius/.store in=\mcRadius, 
pattern radius = 1pt}
\makeatletter
\pgfutil@ifundefined{pgf@pattern@name@_44hh4ngjg}{
\pgfdeclarepatternformonly[\mcThickness,\mcSize]{_44hh4ngjg}
{\pgfqpoint{0pt}{0pt}}
{\pgfpoint{\mcSize+\mcThickness}{\mcSize+\mcThickness}}
{\pgfpoint{\mcSize}{\mcSize}}
{
\pgfsetcolor{\tikz@pattern@color}
\pgfsetlinewidth{\mcThickness}
\pgfpathmoveto{\pgfqpoint{0pt}{0pt}}
\pgfpathlineto{\pgfpoint{\mcSize+\mcThickness}{\mcSize+\mcThickness}}
\pgfusepath{stroke}
}}
\makeatother

 
\tikzset{
pattern size/.store in=\mcSize, 
pattern size = 5pt,
pattern thickness/.store in=\mcThickness, 
pattern thickness = 0.3pt,
pattern radius/.store in=\mcRadius, 
pattern radius = 1pt}
\makeatletter
\pgfutil@ifundefined{pgf@pattern@name@_ahar1zuc8}{
\pgfdeclarepatternformonly[\mcThickness,\mcSize]{_ahar1zuc8}
{\pgfqpoint{0pt}{0pt}}
{\pgfpoint{\mcSize+\mcThickness}{\mcSize+\mcThickness}}
{\pgfpoint{\mcSize}{\mcSize}}
{
\pgfsetcolor{\tikz@pattern@color}
\pgfsetlinewidth{\mcThickness}
\pgfpathmoveto{\pgfqpoint{0pt}{0pt}}
\pgfpathlineto{\pgfpoint{\mcSize+\mcThickness}{\mcSize+\mcThickness}}
\pgfusepath{stroke}
}}
\makeatother

 
\tikzset{
pattern size/.store in=\mcSize, 
pattern size = 5pt,
pattern thickness/.store in=\mcThickness, 
pattern thickness = 0.3pt,
pattern radius/.store in=\mcRadius, 
pattern radius = 1pt}
\makeatletter
\pgfutil@ifundefined{pgf@pattern@name@_7twxray5d}{
\pgfdeclarepatternformonly[\mcThickness,\mcSize]{_7twxray5d}
{\pgfqpoint{0pt}{0pt}}
{\pgfpoint{\mcSize+\mcThickness}{\mcSize+\mcThickness}}
{\pgfpoint{\mcSize}{\mcSize}}
{
\pgfsetcolor{\tikz@pattern@color}
\pgfsetlinewidth{\mcThickness}
\pgfpathmoveto{\pgfqpoint{0pt}{0pt}}
\pgfpathlineto{\pgfpoint{\mcSize+\mcThickness}{\mcSize+\mcThickness}}
\pgfusepath{stroke}
}}
\makeatother

 
\tikzset{
pattern size/.store in=\mcSize, 
pattern size = 5pt,
pattern thickness/.store in=\mcThickness, 
pattern thickness = 0.3pt,
pattern radius/.store in=\mcRadius, 
pattern radius = 1pt}
\makeatletter
\pgfutil@ifundefined{pgf@pattern@name@_8gayv0svk}{
\pgfdeclarepatternformonly[\mcThickness,\mcSize]{_8gayv0svk}
{\pgfqpoint{0pt}{0pt}}
{\pgfpoint{\mcSize+\mcThickness}{\mcSize+\mcThickness}}
{\pgfpoint{\mcSize}{\mcSize}}
{
\pgfsetcolor{\tikz@pattern@color}
\pgfsetlinewidth{\mcThickness}
\pgfpathmoveto{\pgfqpoint{0pt}{0pt}}
\pgfpathlineto{\pgfpoint{\mcSize+\mcThickness}{\mcSize+\mcThickness}}
\pgfusepath{stroke}
}}
\makeatother

 
\tikzset{
pattern size/.store in=\mcSize, 
pattern size = 5pt,
pattern thickness/.store in=\mcThickness, 
pattern thickness = 0.3pt,
pattern radius/.store in=\mcRadius, 
pattern radius = 1pt}
\makeatletter
\pgfutil@ifundefined{pgf@pattern@name@_458lrj2a6}{
\pgfdeclarepatternformonly[\mcThickness,\mcSize]{_458lrj2a6}
{\pgfqpoint{0pt}{0pt}}
{\pgfpoint{\mcSize+\mcThickness}{\mcSize+\mcThickness}}
{\pgfpoint{\mcSize}{\mcSize}}
{
\pgfsetcolor{\tikz@pattern@color}
\pgfsetlinewidth{\mcThickness}
\pgfpathmoveto{\pgfqpoint{0pt}{0pt}}
\pgfpathlineto{\pgfpoint{\mcSize+\mcThickness}{\mcSize+\mcThickness}}
\pgfusepath{stroke}
}}
\makeatother

 
\tikzset{
pattern size/.store in=\mcSize, 
pattern size = 5pt,
pattern thickness/.store in=\mcThickness, 
pattern thickness = 0.3pt,
pattern radius/.store in=\mcRadius, 
pattern radius = 1pt}
\makeatletter
\pgfutil@ifundefined{pgf@pattern@name@_45gu4mu2d}{
\pgfdeclarepatternformonly[\mcThickness,\mcSize]{_45gu4mu2d}
{\pgfqpoint{0pt}{0pt}}
{\pgfpoint{\mcSize+\mcThickness}{\mcSize+\mcThickness}}
{\pgfpoint{\mcSize}{\mcSize}}
{
\pgfsetcolor{\tikz@pattern@color}
\pgfsetlinewidth{\mcThickness}
\pgfpathmoveto{\pgfqpoint{0pt}{0pt}}
\pgfpathlineto{\pgfpoint{\mcSize+\mcThickness}{\mcSize+\mcThickness}}
\pgfusepath{stroke}
}}
\makeatother

 
\tikzset{
pattern size/.store in=\mcSize, 
pattern size = 5pt,
pattern thickness/.store in=\mcThickness, 
pattern thickness = 0.3pt,
pattern radius/.store in=\mcRadius, 
pattern radius = 1pt}
\makeatletter
\pgfutil@ifundefined{pgf@pattern@name@_30pjx5nzl}{
\pgfdeclarepatternformonly[\mcThickness,\mcSize]{_30pjx5nzl}
{\pgfqpoint{0pt}{0pt}}
{\pgfpoint{\mcSize+\mcThickness}{\mcSize+\mcThickness}}
{\pgfpoint{\mcSize}{\mcSize}}
{
\pgfsetcolor{\tikz@pattern@color}
\pgfsetlinewidth{\mcThickness}
\pgfpathmoveto{\pgfqpoint{0pt}{0pt}}
\pgfpathlineto{\pgfpoint{\mcSize+\mcThickness}{\mcSize+\mcThickness}}
\pgfusepath{stroke}
}}
\makeatother
\tikzset{every picture/.style={line width=0.75pt}} 
\scalebox{0.55}{
\begin{tikzpicture}[x=0.75pt,y=0.75pt,yscale=-1,xscale=1]

\draw    (20,20) -- (20,160) ;
\draw    (20,160) -- (80,160) ;
\draw    (80,100) -- (80,160) ;
\draw    (100,80) -- (100,100) ;
\draw    (140,40) -- (140,80) ;
\draw    (120,20) -- (120,40) ;
\draw    (20,20) -- (120,20) ;
\draw    (80,100) -- (100,100) ;
\draw    (100,80) -- (140,80) ;
\draw    (140,40) -- (120,40) ;
\draw    (180,20) -- (180,160) ;
\draw    (180,160) -- (240,160) ;
\draw    (240,100) -- (240,160) ;
\draw    (260,80) -- (260,100) ;
\draw    (300,40) -- (300,80) ;
\draw    (180,20) -- (280,20) ;
\draw    (240,100) -- (260,100) ;
\draw    (260,80) -- (300,80) ;
\draw    (280,40) -- (300,40) ;
\draw    (180,120) -- (200,120) ;
\draw    (280,20) -- (280,40) ;
\draw    (280,40) -- (280,60) ;
\draw    (260,60) -- (260,80) ;
\draw    (220,100) -- (240,100) ;
\draw    (220,100) -- (220,120) ;
\draw    (200,120) -- (220,120) ;
\draw    (260,60) -- (280,60) ;
\draw    (560,120) -- (560,160) ;
\draw    (560,160) -- (620,160) ;
\draw    (620,100) -- (620,160) ;
\draw    (640,80) -- (640,100) ;
\draw    (680,40) -- (680,80) ;
\draw    (620,20) -- (660,20) ;
\draw    (620,100) -- (640,100) ;
\draw    (640,80) -- (680,80) ;
\draw    (660,40) -- (680,40) ;
\draw    (620,20) -- (620,60) ;
\draw    (600,60) -- (620,60) ;
\draw    (600,60) -- (600,80) ;
\draw    (580,80) -- (600,80) ;
\draw    (580,80) -- (580,120) ;
\draw    (560,120) -- (580,120) ;
\draw    (660,20) -- (660,40) ;
\draw    (340,18) -- (340,158) ;
\draw    (340,158) -- (400,158) ;
\draw    (400,98) -- (400,158) ;
\draw    (420,78) -- (420,98) ;
\draw    (460,38) -- (460,78) ;
\draw    (340,18) -- (440,18) ;
\draw    (400,98) -- (420,98) ;
\draw    (420,78) -- (460,78) ;
\draw    (440,38) -- (460,38) ;
\draw    (400,18) -- (400,58) ;
\draw    (380,58) -- (400,58) ;
\draw    (380,58) -- (380,78) ;
\draw    (360,78) -- (380,78) ;
\draw    (360,78) -- (360,118) ;
\draw    (340,118) -- (360,118) ;
\draw    (440,18) -- (440,38) ;
\draw    (440,38) -- (440,58) ;
\draw    (420,58) -- (420,78) ;
\draw    (380,98) -- (400,98) ;
\draw    (380,98) -- (380,118) ;
\draw    (360,118) -- (380,118) ;
\draw    (420,58) -- (440,58) ;
\draw    (560,20) -- (560,160) ;
\draw    (560,20) -- (660,20) ;
\draw    (721,120) -- (721,160) ;
\draw    (721,160) -- (781,160) ;
\draw    (781,100) -- (781,160) ;
\draw    (801,80) -- (801,100) ;
\draw    (841,40) -- (841,80) ;
\draw    (781,20) -- (821,20) ;
\draw    (781,100) -- (801,100) ;
\draw    (801,80) -- (841,80) ;
\draw    (821,40) -- (841,40) ;
\draw    (781,20) -- (781,60) ;
\draw    (761,60) -- (781,60) ;
\draw    (761,60) -- (761,80) ;
\draw    (741,80) -- (761,80) ;
\draw    (741,80) -- (741,120) ;
\draw    (721,120) -- (741,120) ;
\draw    (821,20) -- (821,40) ;
\draw    (821,40) -- (821,60) ;
\draw    (801,60) -- (801,80) ;
\draw    (761,100) -- (781,100) ;
\draw    (761,100) -- (761,120) ;
\draw    (741,120) -- (761,120) ;
\draw    (801,60) -- (821,60) ;
\draw  [draw opacity=0][pattern=_a6umk1f39,pattern size=4.5pt,pattern thickness=0.75pt,pattern radius=0pt, pattern color={rgb, 255:red, 0; green, 0; blue, 0}] (801,60) -- (821,60) -- (821,80) -- (801,80) -- cycle ;
\draw  [draw opacity=0][pattern=_c6rl6fams,pattern size=4.5pt,pattern thickness=0.75pt,pattern radius=0pt, pattern color={rgb, 255:red, 0; green, 0; blue, 0}] (180,120) -- (200,120) -- (200,140) -- (180,140) -- cycle ;
\draw  [draw opacity=0][pattern=_m3vqgrjl0,pattern size=4.5pt,pattern thickness=0.75pt,pattern radius=0pt, pattern color={rgb, 255:red, 0; green, 0; blue, 0}] (200,120) -- (220,120) -- (220,140) -- (200,140) -- cycle ;
\draw  [draw opacity=0][pattern=_pgcud395u,pattern size=4.5pt,pattern thickness=0.75pt,pattern radius=0pt, pattern color={rgb, 255:red, 0; green, 0; blue, 0}] (220,100) -- (240,100) -- (240,120) -- (220,120) -- cycle ;
\draw  [draw opacity=0][pattern=_3xknzeb79,pattern size=4.5pt,pattern thickness=0.75pt,pattern radius=0pt, pattern color={rgb, 255:red, 0; green, 0; blue, 0}] (180,140) -- (200,140) -- (200,160) -- (180,160) -- cycle ;
\draw  [draw opacity=0][pattern=_ya59ch2pd,pattern size=4.5pt,pattern thickness=0.75pt,pattern radius=0pt, pattern color={rgb, 255:red, 0; green, 0; blue, 0}] (200,140) -- (220,140) -- (220,160) -- (200,160) -- cycle ;
\draw  [draw opacity=0][pattern=_6a2p26zyg,pattern size=4.5pt,pattern thickness=0.75pt,pattern radius=0pt, pattern color={rgb, 255:red, 0; green, 0; blue, 0}] (220,140) -- (240,140) -- (240,160) -- (220,160) -- cycle ;
\draw  [draw opacity=0][pattern=_io2cnci7d,pattern size=4.5pt,pattern thickness=0.75pt,pattern radius=0pt, pattern color={rgb, 255:red, 0; green, 0; blue, 0}] (220,120) -- (240,120) -- (240,140) -- (220,140) -- cycle ;
\draw  [draw opacity=0][pattern=_u2a1anb7z,pattern size=4.5pt,pattern thickness=0.75pt,pattern radius=0pt, pattern color={rgb, 255:red, 0; green, 0; blue, 0}] (260,60) -- (280,60) -- (280,80) -- (260,80) -- cycle ;
\draw  [draw opacity=0][pattern=_xe7k92cue,pattern size=4.5pt,pattern thickness=0.75pt,pattern radius=0pt, pattern color={rgb, 255:red, 0; green, 0; blue, 0}] (280,60) -- (300,60) -- (300,80) -- (280,80) -- cycle ;
\draw  [draw opacity=0][pattern=_yfkron34f,pattern size=4.5pt,pattern thickness=0.75pt,pattern radius=0pt, pattern color={rgb, 255:red, 0; green, 0; blue, 0}] (280,40) -- (300,40) -- (300,60) -- (280,60) -- cycle ;
\draw  [draw opacity=0][pattern=_njhkjo5uj,pattern size=4.5pt,pattern thickness=0.75pt,pattern radius=0pt, pattern color={rgb, 255:red, 0; green, 0; blue, 0}] (340,118) -- (360,118) -- (360,138) -- (340,138) -- cycle ;
\draw  [draw opacity=0][pattern=_1p2zk7p6h,pattern size=4.5pt,pattern thickness=0.75pt,pattern radius=0pt, pattern color={rgb, 255:red, 0; green, 0; blue, 0}] (340,138) -- (360,138) -- (360,158) -- (340,158) -- cycle ;
\draw  [draw opacity=0][pattern=_upyub9hd4,pattern size=4.5pt,pattern thickness=0.75pt,pattern radius=0pt, pattern color={rgb, 255:red, 0; green, 0; blue, 0}] (360,138) -- (380,138) -- (380,158) -- (360,158) -- cycle ;
\draw  [draw opacity=0][pattern=_h7get6sa2,pattern size=4.5pt,pattern thickness=0.75pt,pattern radius=0pt, pattern color={rgb, 255:red, 0; green, 0; blue, 0}] (380,138) -- (400,138) -- (400,158) -- (380,158) -- cycle ;
\draw  [draw opacity=0][pattern=_c793v0goe,pattern size=4.5pt,pattern thickness=0.75pt,pattern radius=0pt, pattern color={rgb, 255:red, 0; green, 0; blue, 0}] (360,118) -- (380,118) -- (380,138) -- (360,138) -- cycle ;
\draw  [draw opacity=0][pattern=_dkmkhz45c,pattern size=4.5pt,pattern thickness=0.75pt,pattern radius=0pt, pattern color={rgb, 255:red, 0; green, 0; blue, 0}] (380,118) -- (400,118) -- (400,138) -- (380,138) -- cycle ;
\draw  [draw opacity=0][pattern=_uhj0vhts6,pattern size=4.5pt,pattern thickness=0.75pt,pattern radius=0pt, pattern color={rgb, 255:red, 0; green, 0; blue, 0}] (380,98) -- (400,98) -- (400,118) -- (380,118) -- cycle ;
\draw  [draw opacity=0][pattern=_6g29qa4yr,pattern size=4.5pt,pattern thickness=0.75pt,pattern radius=0pt, pattern color={rgb, 255:red, 0; green, 0; blue, 0}] (420,58) -- (440,58) -- (440,78) -- (420,78) -- cycle ;
\draw  [draw opacity=0][pattern=_ba3143j47,pattern size=4.5pt,pattern thickness=0.75pt,pattern radius=0pt, pattern color={rgb, 255:red, 0; green, 0; blue, 0}] (440,58) -- (460,58) -- (460,78) -- (440,78) -- cycle ;
\draw  [draw opacity=0][pattern=_0knsan5ho,pattern size=4.5pt,pattern thickness=0.75pt,pattern radius=0pt, pattern color={rgb, 255:red, 0; green, 0; blue, 0}] (440,38) -- (460,38) -- (460,58) -- (440,58) -- cycle ;
\draw  [draw opacity=0][pattern=_d2l82x0zk,pattern size=4.5pt,pattern thickness=0.75pt,pattern radius=0pt, pattern color={rgb, 255:red, 0; green, 0; blue, 0}] (380,58) -- (400,58) -- (400,78) -- (380,78) -- cycle ;
\draw  [draw opacity=0][pattern=_ylm12zkue,pattern size=4.5pt,pattern thickness=0.75pt,pattern radius=0pt, pattern color={rgb, 255:red, 0; green, 0; blue, 0}] (400,78) -- (420,78) -- (420,98) -- (400,98) -- cycle ;
\draw  [draw opacity=0][pattern=_wcj7owzr7,pattern size=4.5pt,pattern thickness=0.75pt,pattern radius=0pt, pattern color={rgb, 255:red, 0; green, 0; blue, 0}] (380,78) -- (400,78) -- (400,98) -- (380,98) -- cycle ;
\draw  [draw opacity=0][pattern=_nmj0zh67s,pattern size=4.5pt,pattern thickness=0.75pt,pattern radius=0pt, pattern color={rgb, 255:red, 0; green, 0; blue, 0}] (360,78) -- (380,78) -- (380,98) -- (360,98) -- cycle ;
\draw  [draw opacity=0][pattern=_vdpemwr0m,pattern size=4.5pt,pattern thickness=0.75pt,pattern radius=0pt, pattern color={rgb, 255:red, 0; green, 0; blue, 0}] (360,98) -- (380,98) -- (380,118) -- (360,118) -- cycle ;
\draw  [draw opacity=0][pattern=_dr2197wcf,pattern size=4.5pt,pattern thickness=0.75pt,pattern radius=0pt, pattern color={rgb, 255:red, 0; green, 0; blue, 0}] (580,80) -- (600,80) -- (600,100) -- (580,100) -- cycle ;
\draw  [draw opacity=0][pattern=_3qcqbesgc,pattern size=4.5pt,pattern thickness=0.75pt,pattern radius=0pt, pattern color={rgb, 255:red, 0; green, 0; blue, 0}] (420,18) -- (440,18) -- (440,38) -- (420,38) -- cycle ;
\draw  [draw opacity=0][pattern=_ms4rkwdd5,pattern size=4.5pt,pattern thickness=0.75pt,pattern radius=0pt, pattern color={rgb, 255:red, 0; green, 0; blue, 0}] (400,58) -- (420,58) -- (420,78) -- (400,78) -- cycle ;
\draw  [draw opacity=0][pattern=_vtybccwm3,pattern size=4.5pt,pattern thickness=0.75pt,pattern radius=0pt, pattern color={rgb, 255:red, 0; green, 0; blue, 0}] (400,38) -- (420,38) -- (420,58) -- (400,58) -- cycle ;
\draw  [draw opacity=0][pattern=_hos41ragy,pattern size=4.5pt,pattern thickness=0.75pt,pattern radius=0pt, pattern color={rgb, 255:red, 0; green, 0; blue, 0}] (420,38) -- (440,38) -- (440,58) -- (420,58) -- cycle ;
\draw  [draw opacity=0][pattern=_0wabbbft6,pattern size=4.5pt,pattern thickness=0.75pt,pattern radius=0pt, pattern color={rgb, 255:red, 0; green, 0; blue, 0}] (400,18) -- (420,18) -- (420,38) -- (400,38) -- cycle ;
\draw  [draw opacity=0][pattern=_oqzcc7lxe,pattern size=4.5pt,pattern thickness=0.75pt,pattern radius=0pt, pattern color={rgb, 255:red, 0; green, 0; blue, 0}] (721,120) -- (741,120) -- (741,140) -- (721,140) -- cycle ;
\draw  [draw opacity=0][pattern=_sn8wa2re8,pattern size=4.5pt,pattern thickness=0.75pt,pattern radius=0pt, pattern color={rgb, 255:red, 0; green, 0; blue, 0}] (761,120) -- (781,120) -- (781,140) -- (761,140) -- cycle ;
\draw  [draw opacity=0][pattern=_dt5bp7kd7,pattern size=4.5pt,pattern thickness=0.75pt,pattern radius=0pt, pattern color={rgb, 255:red, 0; green, 0; blue, 0}] (761,100) -- (781,100) -- (781,120) -- (761,120) -- cycle ;
\draw  [draw opacity=0][pattern=_m3ss30n5f,pattern size=4.5pt,pattern thickness=0.75pt,pattern radius=0pt, pattern color={rgb, 255:red, 0; green, 0; blue, 0}] (761,140) -- (781,140) -- (781,160) -- (761,160) -- cycle ;
\draw  [draw opacity=0][pattern=_eeaq4pblw,pattern size=4.5pt,pattern thickness=0.75pt,pattern radius=0pt, pattern color={rgb, 255:red, 0; green, 0; blue, 0}] (741,140) -- (761,140) -- (761,160) -- (741,160) -- cycle ;
\draw  [draw opacity=0][pattern=_buxsekv3w,pattern size=4.5pt,pattern thickness=0.75pt,pattern radius=0pt, pattern color={rgb, 255:red, 0; green, 0; blue, 0}] (741,120) -- (761,120) -- (761,140) -- (741,140) -- cycle ;
\draw  [draw opacity=0][pattern=_f6ixcbm2e,pattern size=4.5pt,pattern thickness=0.75pt,pattern radius=0pt, pattern color={rgb, 255:red, 0; green, 0; blue, 0}] (721,140) -- (741,140) -- (741,160) -- (721,160) -- cycle ;
\draw  [draw opacity=0][pattern=_eh6fw3syn,pattern size=4.5pt,pattern thickness=0.75pt,pattern radius=0pt, pattern color={rgb, 255:red, 0; green, 0; blue, 0}] (821,40) -- (841,40) -- (841,60) -- (821,60) -- cycle ;
\draw  [draw opacity=0][pattern=_d8eo8rhz5,pattern size=4.5pt,pattern thickness=0.75pt,pattern radius=0pt, pattern color={rgb, 255:red, 0; green, 0; blue, 0}] (821,60) -- (841,60) -- (841,80) -- (821,80) -- cycle ;
\draw  [draw opacity=0][pattern=_dfnugp45i,pattern size=4.5pt,pattern thickness=0.75pt,pattern radius=0pt, pattern color={rgb, 255:red, 0; green, 0; blue, 0}] (560,120) -- (580,120) -- (580,140) -- (560,140) -- cycle ;
\draw  [draw opacity=0][pattern=_e8s4tql46,pattern size=4.5pt,pattern thickness=0.75pt,pattern radius=0pt, pattern color={rgb, 255:red, 0; green, 0; blue, 0}] (600,100) -- (620,100) -- (620,120) -- (600,120) -- cycle ;
\draw  [draw opacity=0][pattern=_la3bydh2h,pattern size=4.5pt,pattern thickness=0.75pt,pattern radius=0pt, pattern color={rgb, 255:red, 0; green, 0; blue, 0}] (580,100) -- (600,100) -- (600,120) -- (580,120) -- cycle ;
\draw  [draw opacity=0][pattern=_wj3mg48y5,pattern size=4.5pt,pattern thickness=0.75pt,pattern radius=0pt, pattern color={rgb, 255:red, 0; green, 0; blue, 0}] (600,140) -- (620,140) -- (620,160) -- (600,160) -- cycle ;
\draw  [draw opacity=0][pattern=_ckpacz1d1,pattern size=4.5pt,pattern thickness=0.75pt,pattern radius=0pt, pattern color={rgb, 255:red, 0; green, 0; blue, 0}] (580,140) -- (600,140) -- (600,160) -- (580,160) -- cycle ;
\draw  [draw opacity=0][pattern=_w1w2jqn62,pattern size=4.5pt,pattern thickness=0.75pt,pattern radius=0pt, pattern color={rgb, 255:red, 0; green, 0; blue, 0}] (600,120) -- (620,120) -- (620,140) -- (600,140) -- cycle ;
\draw  [draw opacity=0][pattern=_crybwhpqx,pattern size=4.5pt,pattern thickness=0.75pt,pattern radius=0pt, pattern color={rgb, 255:red, 0; green, 0; blue, 0}] (580,120) -- (600,120) -- (600,140) -- (580,140) -- cycle ;
\draw  [draw opacity=0][pattern=_g0t93lu20,pattern size=4.5pt,pattern thickness=0.75pt,pattern radius=0pt, pattern color={rgb, 255:red, 0; green, 0; blue, 0}] (560,140) -- (580,140) -- (580,160) -- (560,160) -- cycle ;
\draw  [draw opacity=0][pattern=_vqxe0agie,pattern size=4.5pt,pattern thickness=0.75pt,pattern radius=0pt, pattern color={rgb, 255:red, 0; green, 0; blue, 0}] (620,60) -- (640,60) -- (640,80) -- (620,80) -- cycle ;
\draw  [draw opacity=0][pattern=_8df04g6ll,pattern size=4.5pt,pattern thickness=0.75pt,pattern radius=0pt, pattern color={rgb, 255:red, 0; green, 0; blue, 0}] (600,60) -- (620,60) -- (620,80) -- (600,80) -- cycle ;
\draw  [draw opacity=0][pattern=_m4cuigs7k,pattern size=4.5pt,pattern thickness=0.75pt,pattern radius=0pt, pattern color={rgb, 255:red, 0; green, 0; blue, 0}] (620,80) -- (640,80) -- (640,100) -- (620,100) -- cycle ;
\draw  [draw opacity=0][pattern=_x3fbefqtm,pattern size=4.5pt,pattern thickness=0.75pt,pattern radius=0pt, pattern color={rgb, 255:red, 0; green, 0; blue, 0}] (600,80) -- (620,80) -- (620,100) -- (600,100) -- cycle ;
\draw  [draw opacity=0][pattern=_44hh4ngjg,pattern size=4.5pt,pattern thickness=0.75pt,pattern radius=0pt, pattern color={rgb, 255:red, 0; green, 0; blue, 0}] (640,60) -- (660,60) -- (660,80) -- (640,80) -- cycle ;
\draw  [draw opacity=0][pattern=_ahar1zuc8,pattern size=4.5pt,pattern thickness=0.75pt,pattern radius=0pt, pattern color={rgb, 255:red, 0; green, 0; blue, 0}] (660,60) -- (680,60) -- (680,80) -- (660,80) -- cycle ;
\draw  [draw opacity=0][pattern=_7twxray5d,pattern size=4.5pt,pattern thickness=0.75pt,pattern radius=0pt, pattern color={rgb, 255:red, 0; green, 0; blue, 0}] (660,40) -- (680,40) -- (680,60) -- (660,60) -- cycle ;
\draw  [draw opacity=0][pattern=_8gayv0svk,pattern size=4.5pt,pattern thickness=0.75pt,pattern radius=0pt, pattern color={rgb, 255:red, 0; green, 0; blue, 0}] (640,40) -- (660,40) -- (660,60) -- (640,60) -- cycle ;
\draw  [draw opacity=0][pattern=_458lrj2a6,pattern size=4.5pt,pattern thickness=0.75pt,pattern radius=0pt, pattern color={rgb, 255:red, 0; green, 0; blue, 0}] (620,40) -- (640,40) -- (640,60) -- (620,60) -- cycle ;
\draw  [draw opacity=0][pattern=_45gu4mu2d,pattern size=4.5pt,pattern thickness=0.75pt,pattern radius=0pt, pattern color={rgb, 255:red, 0; green, 0; blue, 0}] (620,20) -- (640,20) -- (640,40) -- (620,40) -- cycle ;
\draw  [draw opacity=0][pattern=_30pjx5nzl,pattern size=4.5pt,pattern thickness=0.75pt,pattern radius=0pt, pattern color={rgb, 255:red, 0; green, 0; blue, 0}] (640,20) -- (660,20) -- (660,40) -- (640,40) -- cycle ;
\draw  [dash pattern={on 0.84pt off 2.51pt}]  (721,22) -- (721,120) ;
\draw  [dash pattern={on 0.84pt off 2.51pt}]  (721,20) -- (781,20) ;

\draw (21,172.4) node [anchor=north west][inner sep=0.75pt]    {$\sigma _{Y}( Q)$};
\draw (181,172.4) node [anchor=north west][inner sep=0.75pt]    {$\sigma _{Y}( Q) /_{R} \ \sigma _{Y}( Y)$};
\draw (341,172.4) node [anchor=north west][inner sep=0.75pt]    {$( \sigma _{Y}( Q) /_{R} \ \sigma _{Y}( Y)) /_{R} \ Z$};
\draw (561,172.4) node [anchor=north west][inner sep=0.75pt]    {$\sigma _{Y}( Q) /_{R} \sigma _{Y}( X)$};
\draw (721,172.4) node [anchor=north west][inner sep=0.75pt]    {$\sigma _{Y}( X) /_{R} \ \sigma _{Y}( Y) \ =\ Z$};

\end{tikzpicture}}
    \label{fig:enter-label}
\end{figure}
By this definition the pair $(X,Y)$ is uniquely defined by $(Y,Z)$. Observe that every weak sentence $ X \subseteq_R Q$ appears because, in the cases that $Y=\emptyset$, we have $\sigma_Y(X) = Z$ and $\sigma_Y = id$, so there is an $X = Z$ for all $Z \subseteq_R \sigma_Y(Q) = Q$. Further, for each $X$, there is a pair $(X,Y)$ for each $Y \subseteq_R X$.  
Given how we selected $X$, we now have $sort(Q /_R Y)/_R Z = \sigma_Y(Q) /_R \sigma_Y(X)$ which implies that $sort(sort(Q /_R Y)/_R Z) = sort(\sigma_Y(Q) /_R \sigma_Y(X)) = sort(Q /_R X)$. 
Further, $sort(Z) = sort(\sigma_Y(X) /_R \sigma_Y(Y)) = sort(X /_R Y)$.  
Thus, when we uniquely associate a pair $(Y,Z)$ with a pair $(X,Y)$ we also uniquely associate equivalent summands $h_{sort(sort(Q /_R Y) /_R Z)} \otimes h_{sort(Z)} \otimes h_{sort(Y)}$ and $h_{sort(Q/_R X)} \otimes h_{sort(X /_R Y)} \otimes h_{sort(Y)}$. 
Using this equality, we rewrite \eqref{sumline2} as 
       $$(\Delta \otimes id) \circ \Delta(h_Q) = \sum_{X \subseteq_R Q} \sum_{Y \subseteq_R X} h_{sort(Q /_R X)} \otimes h_{sort(X/_R Y)} \otimes h_{sort(Y)}.$$
This is equivalent to the sum in Equation \eqref{sumline}, so we have shown that $(PSym_A, \Delta, \epsilon)$ is coassociative. To verify the axioms for the counit, we check that the two expressions below are equivalent to $id(h_Q)= h_Q$ and to each other. 
Note that $\mathbb{Q} \otimes PSym_A \cong PSym_A \cong PSym_A \otimes \mathbb{Q}$ so the elements $1 \otimes h_Q$ and $h_Q \otimes 1$ are equivalent to $h_Q$.
    \begin{align*}
        (id \otimes \epsilon) \circ \Delta(h_Q) &= (id \otimes \epsilon)\left(\sum_{J \subseteq_R Q} h_{sort(Q/_R J)} \otimes h_{sort(J)}\right) = h_Q \otimes 1,\\
        (\epsilon \otimes id) \circ \Delta(h_Q) &= (id \otimes \epsilon)\left(\sum_{J \subseteq_R Q} h_{sort(Q/_R J)} \otimes h_{sort(J)}\right) = 1 \otimes h_Q.
    \end{align*}
    Thus, $(PSym_A, \Delta, \epsilon)$ is a coalgebra.
    
    To show that $PSym_A$ is a bialgebra, we must show that the equations in Definition \ref{bialgebra} are satisfied. 
    Here we use the technical notation for multiplication, $\mu: PSym_A \otimes PSym_A \rightarrow PSym_A$ where $\mu(h_P \otimes h_S) = h_{sort(P \cdot S)}$. Also, let $1_P$ denote the multiplicative identity in $PSym_A$ and $T$ denote the map $T(x \otimes y) = y \otimes x$. First, we check that  $\Delta(m(h_P \otimes h_S))= (m \otimes m)(id \otimes T \otimes id)(\Delta \otimes \Delta)(h_P \otimes h_S)$.
For the left-hand side, we have
    $$\Delta(m(h_P \otimes h_S)) = \Delta(h_{sort(P\cdot S)}) = \sum_{J \subseteq_R sort(P \cdot S)} h_{sort(sort(P\cdot S)/_RJ)} \otimes h_{sort(J)}.$$
    For the righthand side, we have
    $$(m \otimes m)(id \otimes T \otimes id)(\Delta \otimes \Delta)(h_P \otimes h_S) =  \cdots$$
\begin{align*}
    &= (m \otimes m)(id \otimes T \otimes id) \sum_{\substack{Y \subseteq_R P\\Z \subseteq_R S}} h_{sort(P /_R Y)} \otimes h_{sort(Y)} \otimes h_{sort(S /_R Z)} \otimes h_{sort(Z)}\\ &= (m \otimes m) \sum_{\substack{Y \subseteq_R P\\Z \subseteq_R S}} h_{sort(P /_R Y)} \otimes h_{sort(S /_R Z)} \otimes h_{sort(Y)} \otimes h_{sort(Z)}\\
    &= \sum_{\substack{Y \subseteq_R P \\ Z \subseteq_R S}} h_{sort(P/_R Y \cdot S/_R Z)} \otimes h_{sort(Y \cdot Z)} = \sum_{\substack{Y \subseteq_R P \\ Z \subseteq_R S}} h_{sort((P \cdot S)/_R  (Y \cdot Z))} \otimes h_{sort(Y \cdot Z)},
    \end{align*}
    where $Y \cdot Z = (y_1, \ldots, y_{\ell(P)}, z_1, \ldots, z_{\ell(S)})$ for $Y = (y_1, y_2, \ldots, y_{\ell(Y)})$ and $z = (z_1, z_2, \ldots, z_{\ell(Z)})$ and empty words are added to the ends of $Y$ and $Z$ so that they are of length $\ell(P)$ and $\ell(S)$ respectively. Let $\sigma$ be a permutation such that $sort(P \cdot S) = \sigma(P \cdot S)$ and note that $sort(\sigma(P \cdot S)/_R \sigma(Y \cdot Z)) = sort((P \cdot S) /_R (Y \cdot Z)$.  
    Rename $\sigma(Y \cdot Z)$ as $J$ and note that $sort(J) = sort(Y \cdot Z)$. Then we can rewrite our sum as
    $$(m \otimes m)(id \otimes T \otimes id)(\Delta \otimes \Delta)(h_P \otimes h_S) = \sum_{J \subseteq_R sort(P \cdot S)} h_{sort(sort(P\cdot S)/_RJ)} \otimes h_{sort(J)},$$
so the equation holds. Next observe that $m \circ (\epsilon \otimes \epsilon)(h_P \otimes h_S) = \epsilon \circ m (h_P \otimes h_S)$ because both sides of the equation equal $1$ if $h_P=h_S=1_P$, and $0$ otherwise. Third, we have for $k \in \mathbb{Q}$ that $(u \otimes u)\Delta(k) = \Delta(u(k)) $ because both sides of the equation equal $1_P \otimes 1_P$. Finally, we have $k = \epsilon \circ u(k) = \epsilon(k \cdot 1_P)= k \epsilon(1_P)= k$. The four properties above confirm that $PSym_A$ is a bialgebra. It is also easily seen to be graded by the size of \psents and connected, so $PSym_A$ is a Hopf algebra by Proposition \ref{bi_to_hopf}.
\end{proof}

\begin{ex} Multiplication and comultiplication on the $h$-basis works as follows,
    $$h_{(aba,c)} h_{(bb,a)} = h_{(aba,bb,a,c)},$$ \vspace{-6mm}
    \begin{multline*}
         \Delta(h_{(ab,bc)}) = h_{(ab,bc)} \otimes 1 + h_{(bc,a)} \otimes h_{(b)} + h_{(ab,b)} \otimes h_{(c)} +  h_{(a,b)} \otimes h_{(b,c)} + \\ h_{(bc)} \otimes h_{(ab)} + h_{(ab)} \otimes h_{(bc)} + h_{(b)} \otimes h_{(ab,c)} + h_{(a)} \otimes h_{(bc,b)} + 1 \otimes h_{(ab,bc)}.
    \end{multline*}
   
\end{ex}

\begin{prop}
    The antipode of $PSym_A$ is given by $$S(h_{P}) = \sum_{J \preceq P}(-1)^{\ell(J)}h_{sort(J)},$$ where the sum runs over sentences $J$ that refine the p-sentence $P$, extended linearly.
\end{prop}

\begin{proof}
   Using the recursive definition of antipode and the comultiplication $\Delta$ of Equation \ref{anti_comult} where $w= a_1 \cdots a_i$ is a word of length $i$, 
   $$S(h_w) = - \sum_{j=0}^{i-1} S(h_{(a_1 \cdots a_j)}) h_{(a_{j+1} \cdots a_i)}.$$ In the case that $i=1$, we have $S(a_1)=-a_1$ which agrees with our formula. Assume that the formula holds for words $w = a_1 \cdots a_i$ where $i=k$. Now let $w = a_1a_2 \cdots a_ka_{k+1}$ be a word. We have 
   \begin{align*}
    S(h_{a_1a_2 \cdots a_ka_{k+1}}) &= - \sum_{j=0}^{k} S(h_{(a_1 \cdots a_j)}) h_{(a_{j+1} \cdots a_{k+1})}  = -h_w - \sum_{j=1}^{k} \sum_{J \preceq (a_1 \cdots a_j)} (-1)^{\ell(J)}h_{sort(J)} h_{(a_{j+1} \cdots a_k)}\\
    &= \sum_{K \preceq w} (-1)^{\ell(K)}h_{sort(K)}.
   \end{align*}
   Then, because the antipode is an anti-endomorphism,
\begin{align*}
     S(h_{P}) &= S(h_{w_k}) \cdots S(h_{w_1}) = \sum_{I_k \preceq w_k}(-1)^{\ell(I_k)}h_{sort(I_k)} \cdots \sum_{I_1 \preceq w_1}(-1)^{\ell(I_1)}h_{sort(I_1)}\\
     &= \sum_{I_k \preceq w_k} \cdots \sum_{I_1 \preceq w_1} (-1)^{\ell(I_k) + \cdots + \ell(I_1)}h_{sort(I_k)} \cdots h_{sort(I_1)} = \sum_{J \preceq P} (-1)^{\ell(J)}h_{sort(J)}. \qedhere
\end{align*}   
\end{proof}

\begin{ex} For the \psent $(aba,c)$, we have 
$$S(h_{(aba,c)})= h_{(aba,c)} - h_{(ba,a,c)} - h_{(ab,a,c)} + h_{(a,a,b,c)}.$$
\end{ex}

We can now show that $PSym_A$ is the commutative image of $NSym_A$. We do so by defining a morphism $\chi$ called the \emph{colored forgetful map} that sends the (noncommutative) generators of $NSym_A$ to the (commutative) generators of $PSym_A$. 

\begin{thm} The algebra homomorphism $\chi: \NSym_A \rightarrow PSym_A$ defined by $H_w \rightarrow h_w$ is a surjective Hopf algebra morphism.
\end{thm}

\begin{proof} It is simple to see that $\chi$ is a surjective algebra homomorphism by definition. For the purposes of this proof, let $\Delta_N$ denote the comultiplication on $NSym_A$ and $\Delta_P$ denote the comultiplicaiton on $PSym_A$. We see that $\chi$ is a coalgebra homomorphism by Equation \eqref{coalg_morph}, as 
\begin{align*}
    (\chi \otimes \chi) \Delta_N(H_I) &= (\chi \otimes \chi) (\sum_{J \subseteq_R I} H_{I/_R J} \otimes H_{J})\\
    &= \sum_{J \subseteq_R I} h_{sort(I/_R J)} \otimes h_{sort(J)} = \Delta_P(h_{sort(I)})= \Delta_P(\chi(H_I))
\end{align*}

Additionally, $\epsilon_{N}(\chi(H_P)) = \epsilon_{P}(H_P)$ because both are $1$ if $P = \emptyset$ and 0 otherwise. Thus, by Proposition \ref{bi_to_hopf}, the map $\chi$ is a Hopf algebra morphism.
\end{proof}

We have established that $PSym_A$ is a Hopf algebra and the commutative image of $NSym_A$, which is analogous to the relationship of $Sym$ and $NSym$.  To show that $PSym_A$ is a colored generalization of $Sym$, we define a map from $PSym_A$ to $Sym$ and show that it is a Hopf isomorphism if $A$ is a unary alphabet.

\begin{defn} Define the \emph{uncoloring map} $$\upsilon: PSym_A \rightarrow Sym \text{\quad by \quad}\upsilon(h_{P}) = h_{w\ell(P)}.$$ 
\end{defn}

\begin{prop} When $A$ is an alphabet of size one, $PSym_A$  and $Sym$ are isomorphic as Hopf algebras.        
\end{prop}

\begin{proof} We show that, in this case, $\upsilon: PSym_A \rightarrow Sym$ is a Hopf isomorphism. Let $A ={a}$ and note that $w\ell$ is now a bijection between partitions and p-sentences because there is only one \psent for each shape $\lambda$.  Let $\lambda_a$ denote the unique p-sentence of shape $\lambda$ and $n_a$ the unique word of length $n$ in the alphabet $A = \{a\}$. Then the formulas for multiplication and comultiplication of the colored complete homogeneous basis simplify to 
    $$h_{\lambda_a}h_{\mu_a} = h_{sort(\lambda_a \cdot \mu_a)},$$
$$h_{n_a} = \sum_{0 \leq i \leq n} h_{i_a} \otimes h_{(n-i)_a}. $$
These are exactly the formulas for the complete homogeneous basis of the symmetric functions from Equations \eqref{h_mult} and \eqref{h_comult} when $\lambda_a$ is replaced with $\lambda$ and so on. From here it is simple to see that $\upsilon$ is a bialgebra isomorphism and thus, by Corollary \ref{hopf_morph}, a Hopf isomorphism.
\end{proof}

One can verify in a similar manner that $\upsilon$ is a Hopf morphism for any $A$. The same is true for $\upsilon: QSym_A \rightarrow QSym$ and $\upsilon: \NSym_A \rightarrow \NSym$.

\section{The colored symmetric functions: $Sym_A$}\label{sec:syma}

We now introduce a second colored generalization of $Sym$ that is a subalgebra of $QSym_A$ and dual to $PSym_A$. As before, we say a monomial $x_{i_1, v_1} \cdots x_{i_j, v_j}$, where $i_1 < \cdots < i_j$, is associated with the sentence $I=(v_1, \cdots, v_j)$. If a monomial is associated with the sentence $I$, we now also associate it with the \psent $sort(I)$.

\begin{defn}
    Let $Sym_A$ denote the set of 
 \emph{colored symmetric functions} $f \in \mathbb{Q}[x_A]$ such that $$f(x_{A,1}, x_{A,2}, \ldots ) = f(x_{A, \sigma(1)}, x_{A, \sigma(2)}, \ldots).$$ In other words, if two monomials in $f$ are associated with the same p-sentence, then they must have the same coefficients.
\end{defn}
\begin{ex} The function  
$$f = x_{a,1}x_{bc,2} + x_{bc,1}x_{a,2} + x_{a,1}x_{bc,3} + x_{bc,1}x_{a,3} + \cdots + x_{a,5}x_{bc,7} + x_{bc,5}x_{a,7} + \cdots $$ is in $Sym_A$ because each term $x_{a,i}x_{bc,j}$ has the same coefficient as the term $x_{a, \sigma(i)}x_{bc, \sigma(j)}$ for any permutation $\sigma$ of $\mathbb{N}$. The following function is not in $Sym_A$:
$$g = x_{a,1}x_{bc,2} + 3x_{bc,1}x_{a,2} + \cdots$$
\end{ex}

 By definition, $Sym_A \subseteq QSym_A$, and in fact $Sym_A$ is a subspace of $QSym_A$.

 \begin{defn} For a \psent $P$, the \emph{colored monomial symmetric function} is defined $$m_P = \sum_{sort(I)=P} M_I,$$ where the sum runs over sentences $I$ such that $sort(I)=P$. Equivalently, if $\ell(P) = k$ and $I = (v_1, \ldots, v_k)$, $$m_P = \sum_{ \substack{ I \in Sent_A \\ sort(I)=P}} \sum_{i_1 < \cdots < i_k} x_{v_1, i_1} \cdots  x_{v_k, i_k} = \sum_{\substack{K \in WSent_A\\sort(K)=P}} x_K.$$ 
 \end{defn}

 \begin{ex} Consider the \psent $(ab,c,c)$. Then, 
 \begin{align*}
     m_{(ab,c,c)} &= M_{(ab,c,c)} + M_{(c,ab,c)} + M_{(c,c,ab)}\\
     &= \sum_{i_1 < i_2 < i_3} x_{ab,i_1}x_{c, i_2}x_{c, i_3} + x_{c, i_1}x_{ab,i_2}x_{c, i_3} + x_{c, i_1}x_{c, i_2}x_{ab,i_3}.
 \end{align*}
 \end{ex}

 \begin{prop}
     The set $\{m_P\}_{P \in PSent_A}$ is a basis of $Sym_A$.
 \end{prop}

 \begin{proof} Consider a colored symmetric function $f = \sum_{K \in WSent_A} b_K x_{K}$ where $b_K$ are rational coefficients.  Since $f \in QSym_A$, we rewrite $f = \sum_{J \in Sent_A} b_J M_J$ because $b_K = b_{\tilde{K}}$ for any weak sentence $K$. By definition of $Sym_A$, if $P = sort(J)$ for a sentence $J$ and \psent $P$, then $b_J = b_P$. Thus, $f= \sum_{P \in PSent_A} b_P m_P$, so the colored monomial symmetric functions span $Sym_A$. Additionally, for any sentence $J$, the monomial $x_J$ only appears in the colored monomial symmetric function $m_{sort(\tilde{K})}$.  It follows that the colored monomial symmetric functions are linearly independent.
 \end{proof}

To show that $Sym_A$ is a Hopf subalgebra of $QSym_A$, it suffices to show that multiplication, comultiplication, and the antipode of $QSym_A$ restrict to $Sym_A$. For p-sentences $P,S, $ and $Q$, let the coefficient $r^{Q}_{P,S}$ denote the number of pairs of weak sentences $Y = (y_1, \ldots, y_m)$ and $Z = (z_1, \ldots, z_m)$ such that $sort(Y)=P$ and $sort(Z)=S$ and $Q = (y_1z_1, y_2z_2, \ldots, y_mz_m)$. Additionally, for \psents $P = (w_1, \ldots, w_{k})$ and $Q = (v_1, \ldots, v_j)$,  write $Q \sqsubseteq P$ if $\{v_1, \ldots, v_j\}$ is a submultiset of $\{w_1, \ldots, w_k \}$. Now, assuming $Q \sqsubseteq P$, define $$P \smallersetminus Q = (u_1, \ldots, u_{k-j}) \in PSent_A$$ such that $\{w_1, \ldots, w_k\} = \{ v_1, \ldots, v_j\} \sqcup \{u_1, \ldots, u_{k-j}\}$ where $\sqcup$ is the union of multisets. For example, $\{aaa, ab, ab\} \sqsubseteq \{aaa, ab, ab, ab, ca, ca \}$,  and $ \{aaa, ab, ab, ab, ca, ca\} \smallersetminus \{aaa, ab, ab\} = \{ ab, ca, ca \}$.

\begin{thm}\label{thm:syma_hopf} $Sym_A$ is a graded Hopf algebra with multiplication given by \[m_Pm_s = \sum_{Q \in PSent_A} r^{Q}_{P,S}m_Q,\] the natural unity map $u(k) = k \cdot 1$, comultiplication given by \[ \Delta(m_P) = \sum_{Q \sqsubseteq P} m_{Q} \otimes m_{P \smallersetminus Q},\] and the counit \[\epsilon(m_P) = \begin{cases}
    1 & \text{ if } P = \emptyset\\
    0 & \text{ otherwise.}
\end{cases}.\]  Moreover, $Sym_A$ is a Hopf subalgebra of $QSym_A$.
\end{thm}
\begin{proof}

First we show for p-sentences $P = (w_1, \ldots, w_k)$ and $S = (v_1, \ldots, v_j)$, that $$m_P m_S = \sum_{Q \in PSent_A} r^{Q}_{P,S} m_Q,$$ where, for $\ell(Q)=m$, the coefficient $r^{Q}_{P,S}$ is the number of pairs of weak sentences $Y = (y_1, \ldots, y_m)$ and $Z = (z_1, \ldots, z_m)$ such that $sort(Y)=P$ and $sort(Z)=S$ and $Q = (y_1z_1, y_2z_2, \ldots, y_mz_m)$, and multiplication is inherited from $QSym_A$.  

    Let $K$ be a weak sentence with $sort(K)=Q$ and $\ell(K) = \ell$. 
    Using the definition of multiplication on $QSym_A$, each time the term $x_K$ appears in the multiplication $m_P m_S$, it is as the multiplication of two monomials $x_I x_J$ where $I = (v_1, \ldots, v_{\ell})$ and $J=(u_1, \ldots, u_{\ell})$ are weak sentences such that $sort(I)=P$ and $sort(J)=S$ and $(v_1u_1, v_2u_2, \ldots, v_{\ell}u_{\ell}) = K$. 
    Let $\sigma_K \in S_{\ell}$ be a permutation such that $\sigma_K(K)=Q$, in other words $(v_{\sigma_K(1)}u_{\sigma_K(1)}, \ldots, v_{\sigma_K(\ell)}u_{\sigma_K(\ell)})=sort(K) \cdot (\emptyset^{\ell - m})=Q \cdot (\emptyset^{\ell - m})$. 
    For any given $(I,J)$ as described above, let $Y = (y_1, \ldots, y_{\ell})= \sigma_K(I)$ and $Z  = (z_1, \ldots, z_{\ell})= \sigma_K(J)$, and observe that we now have $Y, Z$ such that $sort(Y)=P$ and $sort(Z)=S$ and $Q = (y_1z_1, y_2z_2, \ldots, y_mz_m)$ with $(y_{m+1}z_{m+1}, \ldots, y_{\ell}z_{\ell}) = (\emptyset^{\ell - m})$. 
    In this way, each pair $(I,J)$ can be associated with a unique pair $(Y,Z)$ because the pair $(Y,Z)$ is defined by $(I,J)$ by the application of a fixed permutation.  We also obtain every pair $(Y,Z)$ such that $sort(Y)=P$ and $sort(Z)=S$ and $Q = (y_1z_1, y_2z_2, \ldots, y_mz_m)$ because we can do the same logic in reverse. Thus, we have shown that for any weak sentence $K$ such that $sort(K)=Q$, the monomial $x_K$ appears exactly $r^Q_{P,S}$ times in the multiplication of $m_Pm_S$. Now, our first claim follows from the definition of $m_Q$.

 Next, for $P = (w_1, \ldots, w_k)$, we show that $$\Delta(m_P) = \sum_{Q \sqsubseteq P} m_{Q} \otimes m_{P \smallersetminus Q},$$ where $\Delta$ is inherited from $QSym_A$.      
     Observe that 
     \[ \Delta(m_P) = \sum_{sort(I)=P} \Delta(M_I) =  \sum_{sort(I)=P} \sum_{J \cdot K = I} M_J \otimes M_K.\]
          As written, the latter equation sums over all unique rearrangements of $P$ and then splits those rearrangements into two parts. Instead, rewrite the sum to first isolate a part of $P$ and then sum over all unique rearrangements of that part and what remains.
          \begin{align*}
           \Delta(m_P) &= \sum_{Q \sqsubseteq P} \sum_{sort(J)=Q} \sum_{sort(K)=P \smallersetminus Q} M_{J} \otimes M_K = \sum_{Q \sqsubseteq P} \left( \sum_{sort(J)=Q}  M_{J} \right) \otimes \left( \sum_{sort(K)=P \smallersetminus Q} M_K \right)\\
           &= \sum_{Q \sqsubseteq P} m_{Q} \otimes m_{P \smallersetminus Q}.
     \end{align*}
This and the first claim verify that $Sym_A$ is a subcoalgebra and a subalgebra of $QSym_A$, respectively, with the given formulas for multiplication and comultiplication as inherited from $QSym_A$.

Last, we show that $S^*(m_P) \in Sym_A,$ where $S^*$ is the antipode of $QSym_A$.
Observe that
$$S^*(m_P) = \sum_{sort(I)=P} S^*(M_I) = \sum_{sort(I)=P} \sum_{J^r \succeq I} (-1)^{\ell(I)} M_J.$$ Let $K$ and $L$ be two sentences such that $sort(K)=sort(L)$. It suffices to show that $M_K$ and $M_L$ appear the same number of times in the sum above because this implies the sum is a colored symmetric function.  The function $M_K$ will appear once for each unique sentence $I = (w_1, \ldots, w_k)$ such that $sort(I)=P$ and $K^r \succeq I$. In this case, $K^r = (w_1 \cdots w_{i_1}, w_{i_1+1} \cdots w_{i_2}, \cdots)$ for some $1 \leq i_1 < i_2 < \cdots < i_{\ell(K)} \leq k$. For this appearance of $K$, we can see that $L$ appears once as the coarsening of the sentence $I$ obtained by rearranging $(w_1 \cdots w_{i_1}, w_{i_1+1} \cdots w_{i_2}, \cdots)^r$ to equal $L$, then splitting the sentence in between each adjacent $w_j$ and $w_{j+1}$.  This way, we can match each $M_K$ with a unique $M_L$.  We can do the reverse to match each $M_L$ with a unique $M_K$, so there must be the same number of both. This proves our final claim. We have now also shown that the antipode of $QSym_A$ restricts to $Sym_A$, therefore $Sym_A$ is a Hopf subalgebra of $QSym_A$ and so a Hopf algebra in its own right.  \end{proof}

\begin{ex} The multiplication $m_{(bc,a)}m_{(b)}$ expands as
$$m_{(bc,a)}m_{(b)} = m_{(bc,a,b)} + m_{(bcb,a)} + m_{(ab,bc)}.$$
The comultiplication $\Delta( m_{(aba,bb,ca)})$ yields 
\begin{multline*}
    \Delta(m_{(aba,bb,ca)}) = 1 \otimes m_{(aba,bb,ca)} + m_{(aba)} \otimes m_{(bb,ca)} + m_{(bb)} \otimes m_{(aba,ca)} + m_{(ca)} \otimes m_{(aba,bb)} + \\ + m_{(aba,bb)} \otimes m_{(ca)} + m_{(aba,ca)} \otimes m_{(bb)} + m_{(bb,ca)} \otimes m_{(aba)} + m_{(aba,bb,ca)} \otimes 1
\end{multline*}
\end{ex}

Now, we verify that $Sym_A$ is an appropriate analogue to the symmetric functions.

\begin{defn} Define the \emph{uncoloring map} $$\upsilon: Sym_A \rightarrow Sym \text{\quad by \quad}\upsilon(m_{P}) = m_{w\ell(P)}.$$
\end{defn}

 For any $A$, the map $\upsilon$ is a Hopf morphism. This can be verified using the definitions in Section \ref{hopf_alg_sec}. We are again interested in the special case of $|A|=1$ when we have an isomorphism. 

\begin{prop} When $A$ is an alphabet of size one, $Sym_A$ and $Sym$ are isomorphic as Hopf algebras.    
\end{prop}

\begin{proof} We show that, in this case, $\upsilon: Sym_A \rightarrow Sym$ is a Hopf isomorphism.
Let $A ={a}$ and note that $w\ell$ is now a bijection between partitions and \psents because there is only one \psent for each shape $\lambda$.  Let $\lambda_a$ denote the unique p-sentence of shape $\lambda$ and $n_a$ the unique word of length $n$ in the alphabet $A = \{a\}$. Then our formulas for multiplication and comultiplication of the colored monomial basis of $Sym_A$ simplify to 
    $$m_{\lambda_a}m_{\mu_a} = \sum_{\nu_a}r^{\nu_a}_{\lambda_a, \mu_a} m_{\nu_a} = \sum_{\nu_a:\ \nu \in \lambda \Qshuffle \mu} m_{\nu_a},$$
$$m_{\lambda_a} = \sum_{\mu_a \sqsubseteq \lambda_a} m_{\mu_a} \otimes m_{\lambda_a \smallersetminus \mu_a}. $$
These are exactly the formulas for the monomial basis of the symmetric functions from Equations \eqref{mon_mult} and \eqref{mon_comult} when $\lambda_a$ is replaced with $\lambda$ and so on, thus from here it is simple to see that $\upsilon$ is a bialgebra isomorphism and thus a Hopf isomorphism.
\end{proof}

We have now defined two colored generalizations of $Sym$, one specializing $NSym_A$ and one $QSym_A$. These two generalizations are, in fact, dual Hopf algebras. 

\begin{thm} The Hopf algebras $PSym_A$ and $Sym_A$ are dually paired by the inner product $$PSym_A \times Sym_A : \langle \cdot, \cdot \rangle \rightarrow \mathbb{Q} \text{ defined by } \langle h_P, m_S \rangle = \delta_{P,S}.$$
\end{thm}

\begin{proof}
Let $PSym_A^*$ be the graded dual algebra to $PSym_A$ paired by the inner product $\langle \cdot , \cdot \rangle : PSym_A \times PSym_A^* \rightarrow \mathbb{Q}$ defined by $\langle h_P, h_S^* \rangle = \delta_{P,S}$ where $\{h^*_S\}_S$ is the basis dual to $\{h_P\}_P$. Then by Proposition \ref{hopf_dual_coproduct}, the multiplication for this basis is defined as

$$h^*_Qh^*_S = \sum_{P} r^P_{Q,S} h^*_P,$$
where $r^{P}_{Q,S}$ is the number of pairs of weak sentences $Y = (y_1, \ldots, y_m)$ and $Z = (z_1, \ldots, z_m)$ such that $sort(Y)=Q$, $sort(Z)=S$ and $P = (y_1z_1, y_2z_2, \ldots, y_mz_m)$. Comultiplication is expressed as
$$\Delta(h^*_S) = \sum_{P \sqsubseteq S} h^*_{P} \otimes h^*_{S \smallersetminus P}.$$

Notice that these formulas match those in Theorem \ref{thm:syma_hopf} for multiplication and comultiplication of the colored monomial basis of $Sym_A$. It is clear from this definition that there is a bijective bialgebra morphism, and thus a Hopf isomorphism by \cite[Corollary 1.4.27]{grinberg}, between $PSym^*_{A}$ and $Sym_A$ defined by $h^*_P \longleftrightarrow m_p$.
\end{proof}

\begin{prop} The map $\chi: \NSym_A \rightarrow PSym_A$ is adjoint to the inclusion $\iota: Sym_A \hookrightarrow QSym_A$ with respect to the duality pairing.
\end{prop}

\begin{proof} For a \psent $P$ and a sentence $I$, 
    $$\langle \chi(H_{I}), m_P \rangle = \langle h_{sort(I)}, m_P \rangle = \begin{Bmatrix}
        1 & \text{if }sort(I)=P\\
        0 & \text{otherwise}
    \end{Bmatrix}  = \left\langle H_{I}, \sum_{sort(J)=P} M_J \right\rangle = \langle H_I, \iota(m_P) \rangle.$$ Thus by Equation \eqref{adjoint}  we have shown the maps are adjoint.
\end{proof}

\begin{rem}
As a commutative version of $NSym_A$, the new algebra $PSym_A$ is isomorphic to the Hopf algebra of decorated rooted tall trees (or rooted colored tall trees) contained in the decorated $CK$.  Likewise, $Sym_A$ as a Hopf subalgebra of $QSym_A$ can be formulated in terms of decorated rooted tall trees within decorated $NCK$. While decorated $CK$ and decorated $NCK$ are not dual, the subalgebras $PSym_A$ and $Sym_A$ are, and as such one could situate $Sym_A$ inside decorated $GL$, the graded dual of decorated $CK$, instead.
\end{rem}

Figure \ref{fig:alg_diag} depicts the relationships between $Sym, \NSym, QSym, PSym_A, Sym_A, \NSym_A,$ and $QSym_A$. Dual algebras are connected by dashed lines.

\begin{figure}[h!]
\centering
\tikzset{every picture/.style={line width=0.75pt}} 
\begin{tikzpicture}[x=0.75pt,y=0.75pt,yscale=-1,xscale=1]
\draw    (39.67,96) -- (88.02,129.53) ;
\draw [shift={(89.67,130.67)}, rotate = 214.73] [color={rgb, 255:red, 0; green, 0; blue, 0 }  ][line width=0.75]    (10.93,-3.29) .. controls (6.95,-1.4) and (3.31,-0.3) .. (0,0) .. controls (3.31,0.3) and (6.95,1.4) .. (10.93,3.29)   ;
\draw    (39.67,70.67) -- (88.1,32.24) ;
\draw [shift={(89.67,31)}, rotate = 141.57] [color={rgb, 255:red, 0; green, 0; blue, 0 }  ][line width=0.75]    (10.93,-3.29) .. controls (6.95,-1.4) and (3.31,-0.3) .. (0,0) .. controls (3.31,0.3) and (6.95,1.4) .. (10.93,3.29)   ;
\draw    (270,30) -- (221.56,68.75) ;
\draw [shift={(220,70)}, rotate = 321.34] [color={rgb, 255:red, 0; green, 0; blue, 0 }  ][line width=0.75]    (10.93,-3.29) .. controls (6.95,-1.4) and (3.31,-0.3) .. (0,0) .. controls (3.31,0.3) and (6.95,1.4) .. (10.93,3.29)   ;
\draw    (220.67,94) -- (268.38,128.82) ;
\draw [shift={(270,130)}, rotate = 216.12] [color={rgb, 255:red, 0; green, 0; blue, 0 }  ][line width=0.75]    (10.93,-3.29) .. controls (6.95,-1.4) and (3.31,-0.3) .. (0,0) .. controls (3.31,0.3) and (6.95,1.4) .. (10.93,3.29)   ;
\draw    (140.67,31) -- (190.08,68.79) ;
\draw [shift={(191.67,70)}, rotate = 217.41] [color={rgb, 255:red, 0; green, 0; blue, 0 }  ][line width=0.75]    (10.93,-3.29) .. controls (6.95,-1.4) and (3.31,-0.3) .. (0,0) .. controls (3.31,0.3) and (6.95,1.4) .. (10.93,3.29)   ;
\draw    (319.67,30) -- (368.1,68.75) ;
\draw [shift={(369.67,70)}, rotate = 218.66] [color={rgb, 255:red, 0; green, 0; blue, 0 }  ][line width=0.75]    (10.93,-3.29) .. controls (6.95,-1.4) and (3.31,-0.3) .. (0,0) .. controls (3.31,0.3) and (6.95,1.4) .. (10.93,3.29)   ;
\draw    (140.67,130) -- (188.07,94.21) ;
\draw [shift={(189.67,93)}, rotate = 142.94] [color={rgb, 255:red, 0; green, 0; blue, 0 }  ][line width=0.75]    (10.93,-3.29) .. controls (6.95,-1.4) and (3.31,-0.3) .. (0,0) .. controls (3.31,0.3) and (6.95,1.4) .. (10.93,3.29)   ;
\draw    (369.67,97) -- (322.66,128.88) ;
\draw [shift={(321,130)}, rotate = 325.86] [color={rgb, 255:red, 0; green, 0; blue, 0 }  ][line width=0.75]    (10.93,-3.29) .. controls (6.95,-1.4) and (3.31,-0.3) .. (0,0) .. controls (3.31,0.3) and (6.95,1.4) .. (10.93,3.29)   ;
\draw  [dash pattern={on 4.5pt off 4.5pt}]  (150,20) -- (270,20) ;
\draw  [dash pattern={on 4.5pt off 4.5pt}]  (147,143) -- (267,143) ;
\draw  [dash pattern={on 4.5pt off 4.5pt}]  (30,100) -- (30,170) ;
\draw  [dash pattern={on 4.5pt off 4.5pt}]  (30,170) -- (380.67,170) ;
\draw  [dash pattern={on 4.5pt off 4.5pt}]  (380,100) -- (380,170) ;
\draw (88.67,9.73) node [anchor=north west][inner sep=0.75pt]    {$PSym_{A}$};
\draw (11,72.4) node [anchor=north west][inner sep=0.75pt]    {$NSym_{A}$};
\draw (189.67,73.73) node [anchor=north west][inner sep=0.75pt]    {$ \begin{array}{l}
Sym\\
\end{array}$};
\draw (94,133.4) node [anchor=north west][inner sep=0.75pt]    {$NSym$};
\draw (272.67,133.07) node [anchor=north west][inner sep=0.75pt]    {$QSym$};
\draw (345,73.4) node [anchor=north west][inner sep=0.75pt]    {$QSym_{A}$};
\draw (272.67,9.73) node [anchor=north west][inner sep=0.75pt]    {$Sym_{A}$};
\draw (173,107.4) node [anchor=north west][inner sep=0.75pt]    {$\chi $};
\draw (47,32.4) node [anchor=north west][inner sep=0.75pt]    {$\chi $};
\draw (50,115.4) node [anchor=north west][inner sep=0.75pt]    {$\upsilon $};
\draw (172,32.4) node [anchor=north west][inner sep=0.75pt]    {$\upsilon $};
\draw (231,32.4) node [anchor=north west][inner sep=0.75pt]    {$\upsilon $};
\draw (350,114.4) node [anchor=north west][inner sep=0.75pt]    {$\upsilon $};
\draw (349,31.4) node [anchor=north west][inner sep=0.75pt]    {$i$};
\draw (225,110.4) node [anchor=north west][inner sep=0.75pt]    {$i$};
\end{tikzpicture}
    \caption{Relationships between algebras.}\label{algebras_diag}
    \label{fig:alg_diag}
\end{figure}
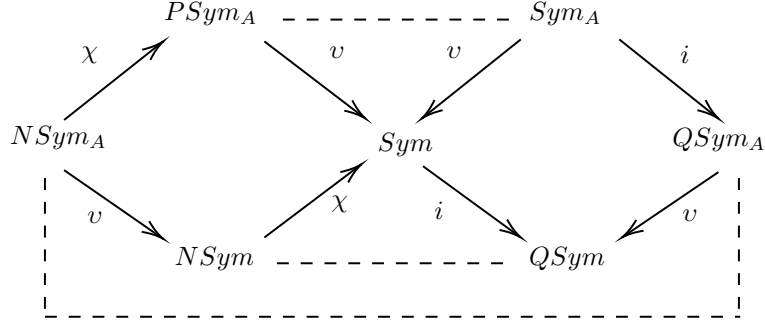

\begin{prop}
    The diagram in Figure \ref{fig:alg_diag} is commutative. 
\end{prop}

\begin{proof} To verify commutativity we must check that $\upsilon(\chi(f))= \chi( \upsilon(f))$ for $f\in \NSym_A$ and $i(\upsilon(g))=\upsilon(i(g))$ for $g \in Sym_A$. Consider some $f\in \NSym_A = \sum_{I \in Sent_A} b_{I} H_I$. Then, $$\chi(f) = \sum_{P \ in PSent_A} ( \sum_{I : sort(I)=P} b_I) h_P \in PSym_A \text{\quad and so \quad}\upsilon(\chi(f)) = \sum_{\lambda} (\sum_{I : w\ell(sort(I))=\lambda} b_I ) h_{\lambda}.$$ On the other hand, $$\upsilon(f) = \sum_{\alpha} (\sum_{I : w\ell(I) = \alpha} b_I) h_{\alpha} \text{\quad and so \quad }\chi(\upsilon(f)) = \sum_{\lambda} (\sum_{I : sort(w\ell(I))=\lambda} b_I) h_{\lambda}.$$ This verifies commutativity of the left portion of the figure.
Next, consider $g = \sum_{P \in PSent_A} d_{P} m_{P} \in Sym_A$. Observe that $$\upsilon(g) = \sum_{\lambda} (\sum_{P : w\ell(P) =\lambda} d_P) m_{\lambda}\text{\quad and so \quad} i(\upsilon(g)) = \sum_{\alpha} (\sum_{I : w\ell(I) = \alpha} d_{sort(I)}) M_{\alpha}.$$  On the other hand, $$i(g) = \sum_{I} d_{sort(I)} M_I \text{\quad and so \quad} \upsilon(i(g)) = 
\sum_{\alpha} (I :\sum_{w\ell(I)=\alpha} d_{sort(I)}) M_{\alpha}.$$ This verifies commutativity of the right side of the diagram.
\end{proof}

\section{The colored Schur and colored dual Schur functions}\label{sec:colorschur}

One generalization of the Schur basis to $PSym_A$ and $Sym_A$ is fairly immediate using colored versions of semistandard Young tableaux. For a sentence $J = (w_1, \ldots, w_k)$, the \textit{colored composition diagram} of shape $J$ is a composition diagram of $w\ell(J)$ where the $j^{\text{th}}$ box in row $i$ is colored with the $j^{\text{th}}$ color in $w_i$.

\begin{defn}
A \emph{colored semistandard Young tableau} of shape $P$ is a colored diagram of $P$ filled with positive integers such that each row is weakly increasing from left to right and each column is strictly increasing from top to bottom. 
\end{defn}

The \emph{type} of a CSSYT $T$ is a sentence $B=(u_1,\ldots ,u_j)$ that indicates how many boxes of each color are filled with each integer and in what order those boxes appear. That is, each word $u_i$ in $B$ is defined by starting in the lowest box containing an $i$ and reading the colors of all boxes containing $i$'s going from left to right, bottom to top. If no box is filled with the number $i$, then $u_i=\emptyset$. For a CIT $T$ of type $B=(u_1, \ldots, u_j)$, the monomial $x_T$ is defined $x_T = x_{u_1,1}x_{u_2,2} \cdots x_{u_j,j}$, which may also be denoted $x_B$.
 
The colored Kostka number, $\mathcal{K}_{P,J}$ denotes the number of colored semistandard Young tableaux of shape $P$ and type $J$.

\begin{defn} For a \psent $P$, the \emph{colored dual Schur function} is defined as $$s^*_{P} = \sum_{Q \in PSent_A} \mathcal{K}_{P,Q} m_Q.$$ 
\end{defn}

\begin{ex} The only CSSYT of shape $(abb,ca)$ whose types are \psents are given below, so we have the following colored dual Schur function.
\ytableausetup{boxsize=8mm}
 $$s^*_{(abb,ca)} = m_{(abb,ca)} + m_{(ab,cb,a)} + m_{(ab,ca,b)}$$
    $$\scalebox{.85}{
    \begin{ytableau}
        a, 1 & b, 1 & b, 1 \\
        c, 2 & a, 2
    \end{ytableau}
    \qquad \quad
    \begin{ytableau}
        a, 1 & b, 1 & b, 2 \\
        c, 2 & a, 3
    \end{ytableau}
    \qquad \quad 
    \begin{ytableau}
        a, 1 & b, 1 & b, 3 \\
        c, 2 & a, 2
    \end{ytableau}
    }$$\vspace{0mm}
    
In most cases, $s^*_P$ can not be expanded into a sum of monomials $x_T$ associated with all of the colored semistandard Young tableaux of shape $P$. For example, in the function $$s^*_{(a,ba)} = m_{(a,ba)} + m_{(a,b,a)} = x_{a,1}x_{ba,2} + x_{ba,1}x_{a,2} +  x_{a,1}x_{b,2}x_{a,3} + \cdots,$$ there is a term associated with the sentence $(ba,a)$ which cannot possibly be the type of a colored Young tableau of shape $(a,ba)$.
\end{ex}

\begin{prop}
    The colored dual Schur functions $\{ s^*_{P}\}_P$ form a basis of $Sym_A$.
\end{prop}

\begin{proof}
    Consider the transition matrix between the colored dual Schur basis and the colored monomial symmetric basis where the indices are ordered first by reverse lexicographic order on partitions applied to the $w\ell$ of the sentences and second by the lexicographic order on words. For example, the sentence $(ab,c)$ comes after $(abcb,c)$ but before $(cc,b)$. We show that this matrix is upper unitriangular. The entries on the diagonal, $\mathcal{K}_{P,P}$, are always greater than 0 because there always exists a CSSYT of shape $P$ and type $P$ given by filling each box in row $i$ with $i$'s. It remains to show that each entry below the diagonal is 0.  Observe that because every CSSYT is a colored immaculate tableau (one in which only the first column need strictly increase \cite{Dau23}), the number of CSSYT of a certain shape and type is always less than the number of colored immaculate tableau of that shape and type. We have chosen our order so that any entry below the diagonal in this matrix corresponds to an entry below the diagonal in the colored immaculate Kostka matrix used in \cite[Theorem 4.21]{Dau23}. We showed in the proof of \cite[Theorem 4.21]{Dau23} that the entries below the diagonal in the colored immaculate Kostka matrix are zero, and so the entries below the diagonal in the colored Schur Kostka matrix are zero as well. Therefore, the transition matrix is upper unitriangular and the colored dual Schur functions form a basis.
\end{proof}

\begin{defn}
    For a \psent $P$, the \emph{colored Schur function} $s_P$ is defined by $$\langle s_P, s^*_{Q} \rangle = \delta_{P,Q},$$ for all \psents  $Q$. In other words, $\{s_P\}_P$ is defined as the basis of $PSym_A$ that is dual to the colored dual Schur basis of $Sym_A$.  
\end{defn}

\begin{prop} Let $A$ be a unary alphabet. Then, for \psents $P$ and $Q$, 
$$\upsilon(s^*_P) = s_{w\ell(P)} \text{\qquad and \qquad} \upsilon(s_Q) = s_{w\ell({Q})}.$$
\end{prop}

\begin{proof}
First, we show that $\upsilon(s^*_{P})= s_{w\ell(P)}$. Observe that when $|A|=1$, colored semistandard Young tableaux are in clear bijection with semistandard Young tableaux and so $\mathcal{K}_{P,Q} = K_{w\ell(P), w\ell(Q)}$. Additionally, there is only one sentence $P$ such that $w\ell(P)= \lambda$ for any partition $\lambda$.  Then, if $w\ell(P) = \lambda$, we have $$\upsilon(s^*_{P}) = \sum_Q \mathcal{K}_{P,Q} \upsilon(m_Q) = \sum_{Q} K_{w\ell(P), w\ell(Q)} m_{w\ell(Q)} = \sum_{\mu} K_{\lambda, \mu} m_{\mu} = s_{\lambda}.$$

Since the matrices for $\mathcal{K}_{P,Q}$ and $K_{\lambda, \mu}$ are equal, so are their inverse matrices meaning $\mathcal{K}^{-1}_{P,Q} = K^{-1}_{w\ell(P), w\ell(Q)}.$ Therefore, if $w\ell(Q) = \mu$, $$\upsilon(s_Q) = \sum_{P} K^{-1}_{w\ell(P), w\ell(Q)} \upsilon(h_{P}) = \sum_{P} K^{-1}_{w\ell(P), w\ell(Q)} h_{w\ell(P)} = \sum_{\lambda} K^{-1}_{\lambda, \mu}h_{\lambda} = s_{\mu}. \qedhere$$
\end{proof}

\begin{ex}
    The colored Schur function for $(aaa,aa)$ is given by $$s^*_{(aaa,aa)} = m_{(aaa,aa)} + 2 m_{(aa,aa,a)} + 3m_{(aa,a,a,a)} + 5m_{(a,a,a,a,a)}.$$ Applying the uncoloring map, we get $$\upsilon(s^*_{(aaa,aa)}) = m_{(3,2)} + 2 m_{(2,2,1)} + 3m_{(2,1,1,1)} + 5m_{(1,1,1,1,1)} = s_{(3,2)}.$$
\end{ex}

\section{Future work}

It remains open to study further properties of these generalizations of the Schur functions such as Pieri rules, creation operator expressions, and Jacobi-Trudi rules. Additionally, in \cite{Dau23, Dau_thesis}, the author introduced the colored immaculate, colored young noncommutative Schur, and colored shin bases of $NSym_A$, as well as their duals: the colored dual immaculate, the colored young quasisymmetric Schur, and the colored extended Schur bases of $QSym_A$.  Each of these specializes to a Schur-like basis of $NSym$ or $QSym$, which generalizes the Schur bases of $Sym$.  

A variety of interesting questions remain about $Sym_A$ and $PSym_A$ in relation to these colored Schur-like bases, including the following: How do the colored Schur and colored dual Schur bases relate to the various colored Schur-like bases of $NSym_A$ and $QSym_A$? What are the commutative images of the colored immaculate, colored Young noncommutative Schur, and colored shin bases? Do any subsets of these images form bases of $PSym_A$ that generalize the Schur functions? Where is there overlap between these images, if any? Are there generalizations of the Schur functions in $Sym_A$ that we can express as sums of the colored dual immaculate, colored Young quasisymmetric Schur, and colored extended Schur functions in meaningful combinatorial ways?

\printbibliography

\end{document}